\documentclass[11pt,reqno]{amsart} 

\usepackage[left=20mm,right=20mm,top=20mm,bottom=20mm]{geometry}
\usepackage{graphicx} 
\usepackage{array} 
\usepackage{paralist} 
\usepackage{amsmath,amsfonts,amsthm,amssymb}
\usepackage{cite}
\usepackage[round]{natbib}

\theoremstyle{plain}
\newtheorem{theorem}{Theorem}[section]
\newtheorem{lemma}[theorem]{Lemma}
\newtheorem{proposition}[theorem]{Proposition}
\newtheorem{corollary}[theorem]{Corollary}

\theoremstyle{definition}
\newtheorem{definition}[theorem]{Definition}
\newtheorem{assumption}[theorem]{Assumption}
\newtheorem{notation}[theorem]{Notation}

\theoremstyle{remark}
\newtheorem{remark}[theorem]{Remark}
\newtheorem{example}[theorem]{Example}

\makeatletter
\@addtoreset{equation}{section}

\makeatother

\begin{document}

\title[Lead-lag estimation between two processes driven by fBM]
{Estimation of the lead-lag parameter between two stochastic processes driven by fractional Brownian motions}
\author{Kohei Chiba}
\address{Graduate School of Mathematical Sciences, the University of Tokyo, Komaba, Tokyo,
153-8914, Japan.}
\email{kchiba@ms.u-tokyo.ac.jp}
\date{} 

\begin{abstract}
In this paper, we consider the problem of estimating the lead-lag parameter between two stochastic processes driven by fractional Brownian motions (fBMs) of the Hurst parameter greater than 1/2. First we propose a lead-lag model between two stochastic processes involving fBMs, and then construct a consistent estimator of the lead-lag parameter with possible convergence rate. Our estimator has the following two features. Firstly, we can construct the lead-lag estimator without using the Hurst parameters of the underlying fBMs. Secondly, our estimator can deal with some non-synchronous and irregular observations. We explicitly calculate possible convergence rate when the observation times are (1) synchronous and equidistant, and (2) given by the Poisson sampling scheme. We also present numerical simulations of our results using the R package YUIMA.
\end{abstract}

\maketitle

\section{Introduction}\label{section1}

Lead-lag effect is phenomenon that some asset prices follow the fluctuation of others with a small time lag. Recently, milli-second level high frequency trading data are available, and the lead-lag effect is observed at such a fine time scale. For a phenomenological perspective, see \cite{huth2014high} and references therein. If we can effectively use the lead-lag effect, it ultimately leads to the prediction of future behavior of stock prices. Hence it is important to analyse the lead-lag effect for developing a trading strategy. 

In order to analyse the lead-lag effect statistically, \cite{hoffmann2013estimation} introduced a \textit{regular semimartingale with the lead-lag parameter $\theta\in(-\delta,\delta)$} for $\delta>0$, and constructed a consistent estimator for the lead-lag parameter $\theta$. Here we briefly review their framework. Before describing the model, we introduce some notations that are used throughout this paper. We follow the notations used in \cite{hoffmann2013estimation}.

\begin{notation}
Let $T>0$ be a terminal time and $\delta>0$ the maximum temporal lead-lag allowed for the model. We set $\Theta=(-\delta,\delta)$. The underlying probability space is denoted by $(\Omega,\mathcal{F},\mathbb{P})$, and we consider a filtration $\mathbb{F}=(\mathcal{F}_{t})_{t\in[-\delta,T+\delta]}$ satisfying the usual conditions on $(\Omega,\mathcal{F},\mathbb{P})$. We set $\tau_{-s}(X)_{t}=X_{t+s}$ for a stochastic process $X$ and a real number $s\in\mathbb{R}$. For a subinterval $[a,b]$ of $[-\delta,T+\delta]$, we set $ \mathbb{F}_{[a,b]} = (\mathcal{F}_{t})_{t\in[a,b]} $. Finally, for a lead-lag parameter $\theta\in(-\delta,\delta)$, the shifted filtration $(\mathcal{F}_{t-\theta})_{t\in[-\delta+\theta,T+\delta+\theta]}$ is denoted by $ \mathbb{F}^{\theta} $.
\end{notation}

They defined a \textit{regular semimartingale with the lead-lag parameter} $\theta$ as follows.

\begin{definition}\label{llag1}
The two-dimensional stochastic process $\{(X_{t},Y_{t})\}_{t\in[0,T+\delta]}$ is called a regular semi-martingale with the lead-lag parameter $\theta\in[0,\delta)$ if X and Y admit the decomposition $X=X^{c} + A$ and $Y=Y^{c} + B$ respectively, where 
\begin{itemize}
\item[(1)] the process $X^{c}$ is a continuous $\mathbb{F}$-local martingale, and the process $Y^{c}$ is a continuous $\mathbb{F}^{\theta}_{[0,T+\delta]}$-local martingale,
\item[(2)] the quadratic variation processes $\langle X^{c} \rangle$ and $\langle Y^{c} \rangle$ are absolutely continuous with respect to the Lebesgue measure on $[0,T+\delta]$, and their Radon-Nikodym derivatives have a locally bounded version, and
\item[(3)] the processes $A$ and $B$ are of finite variation on $[0,T+\delta]$.
\end{itemize}
\end{definition}

\begin{definition}
If $\theta\in(-\delta,0]$ and the two-dimensional process (Y,X) is a regular semimartingale with the lead-lag parameter $-\theta$, then the process $(X,Y)$ is called a regular semimartingale with the lead-lag parameter $\theta\in(-\delta,0]$. 
\end{definition}

Let us set $\mathcal{T}_{n}^{X} = \{ 0=s_{0}^{n} < s_{1}^{n} <\cdots < s_{N_{X}(n)}^{n} =T+\delta \}$ and $\mathcal{T}_{n}^{Y} = \{ 0=t_{0}^{n} < t_{1}^{n} < \cdots < t_{N_{Y}(n)}^{n}=T+\delta \}$ for some positive integers $N_{X}(n)$ and $N_{Y}(n)$. The process $X$ is assumed to be observed at the times in $\mathcal{T}_{n}^{X}$ and the process $Y$ at the times in $\mathcal{T}_{n}^{Y}$. Note that the observation times may be non-synchronous, irregular and random. Since we keep in mind high-frequency data, we consider the asymptotics such that the maximum mesh size of the observation times tend to $0$ as $n\to\infty$:
\[
\lim_{n\to\infty}\max\{|s_{i}^{n}-s_{i-1}^{n}|,\ |t_{j}^{n}-t_{j-1}^{n}| \mid i=1,\ldots,N_{X}(n),\ j=1,\ldots,N_{Y}(n)\} = 0
\]
in probability. They constructed an estimator $\hat{\theta}_{n}$ for $\theta$ by maximizing the \textit{shifted Hayashi-Yoshida covariation contrast} as follows. We define the \textit{shifted Hayashi-Yoshida covariation contrast} $\mathcal{U}_{n}(\tilde{\theta})$ by 
\begin{align*}
\mathcal{U}_{n}(\tilde{\theta}) &= \mathbf{1}_{\tilde{\theta}\geq0} \sum_{i,j:s_{i}^{n}\leq T} (X_{s_{i}^{n}} - X_{s_{i-1}^{n}})(Y_{t_{j}^{n}} - Y_{t_{j-1}^{n}}) \mathbf{1}_{(s_{i-1}^{n},s_{i}^{n}]\cap(t_{j-1}^{n}-\tilde{\theta},t_{j}^{n}-\tilde{\theta}]\neq \emptyset} \\
&\quad + \mathbf{1}_{\tilde{\theta}<0} \sum_{i,j:t_{j}^{n}\leq T} (X_{s_{i}^{n}} - X_{s_{i-1}^{n}})(Y_{t_{j}^{n}} - Y_{t_{j-1}^{n}})\mathbf{1}_{(s_{i-1}^{n},s_{i}^{n}]\cap(t_{j-1}^{n}+\tilde{\theta},t_{j}^{n}+\tilde{\theta}]\neq \emptyset}.
\end{align*}
The lead-lag estimator is constructed by maximizing the contrast $\tilde{\theta} \mapsto |\mathcal{U}_{n}(\tilde{\theta})|$ over a finite subset $\mathcal{G}^{n}$ of $\Theta$. More precisely, the estimator $\hat{\theta}_{n}$ is defined by a solution of the equation
\[
\hat{\theta}_{n} = \max_{\tilde{\theta}\in\mathcal{G}^{n}} |\mathcal{U}_{n}(\tilde{\theta})|.
\]
They proved that the estimator $\hat{\theta}_{n}$ is consistent under appropriate assumptions on $\mathcal{T}_{n}^{X}$, $\mathcal{T}_{n}^{Y}$ and $\mathcal{G}^{n}$ (see \cite{hoffmann2013estimation} for detail).

In standard financial theory, it is assumed that there is no arbitrage in the market. Hence semimartingales, which satisfy this no-arbitrage assumption, are regarded as reasonable class for modeling stock prices. However, non-semimartingales, especially ones involving fractional Brownian motion (fBM), also attract attention as a model of stock prices recently. For example, many researchers investigate how arbitrage opportunities arise and how can we exclude them when we use non-semimartingales as a model of stock prices (\cite{bender2008pricing, cheridito2003arbitrage, guasoni2006no, jarrow2009no, bender2011fractional} and \cite{rogers1997arbitrage} to name but a few). Using fBM in modeling stock prices reflects possible long range dependence property empirically obeserved in some financial time series (for example \cite{cutland1995stock, greene1977long, hall2000semiparametric, henry2002long, lo1991long,  teverovsky1999critical} and \cite{willinger1999stock}). An example of such models is the \textit{fractional Black-Scholes model}. This model assumes that the stock price dynamics are given by
\begin{align}\label{fBS2}
dS_{t} = \mu S_{t}\ dt + \sigma S_{t} dB_{t},\ S_{0}>0,\ t\geq 0
\end{align}
where $\mu$ and $\sigma$ are constants and the process $B=(B_{t})_{t\geq0}$ is an fBM with Hurst parameter $H\in(1/2,1)$. Here the stochastic integral $ \sigma \int_{0}^{t}S_{s}dB_{s} $ can be understood in various ways: if we use the Riemann-Stieltjes integral, then the solution $S^{(1)}$ of the equation (\ref{fBS2}) is
\begin{align*}
S_{t}^{(1)} = S_{0} \exp( \mu t + \sigma B_{t} ),
\end{align*}
and on the other hand, if we use the Wick-It\^{o}-Skorokhod integral, then the solution $S^{(2)}$ of the equation (\ref{fBS2}) becomes
\begin{align*}
S_{t}^{(2)} = S_{0} \exp\left( \mu t - \frac{\sigma^{2}}{2} t^{2H} + \sigma B_{t}  \right)
\end{align*}
(for reference, see \cite{biagini2008stochastic, mishura2008stochastic} or \cite{sottinen2003arbitrage}).

Now a natural question arises: if we consider estimation of the lead-lag parameter between non-semimartingales, especially ones involving fBM, then can we construct a consistent estimator for the lead-lag parameter $\theta$?

Motivated by this quastion, we consider the problem of estimating the lead-lag parameter $\theta$ between two stochastic processes driven by fBMs in this paper. The main aim of this paper is to construct a consistent estimator of the lead-lag parameter $\theta$ with possible convergence rate. 

Although our analysis is based on the framework by \cite{hoffmann2013estimation}, there are different types of frameworks for lead-lag estimation. For example, \cite{robert2010limiting} studied the lead-lag phenomenon using random matrix theory, and \cite{koike2016asymptotic} investigated the asymptotic structure of the likelihood ratio process when the observed processes are correlated Brownian motions with the microstructure noise. 

The rest of this paper is organized as follows. In section 2, we give preliminary results about fractional calculus and stochastic calculus involving fractional Brownian motion, which will be used in subsequent sections of this paper. We also define a covariance structure between two fBMs in Section 2. In section 3, we describe our model. In Section 4, we construct an estimator of the lead-lag parameter $\theta$. We prove its consistency in Section 5. In Section 6, we give an example of our model and a numerical simulation.

\section{Preliminaries}

\subsection{Tools from fractional calculus and stochastic calculus}

In this section, we collect some results about fBM that are used in this paper. Then we introduce a correlation between two fBMs. Let us define a fBM at first.

\begin{definition}
Let $T>0$ be a positive number. A centered Gaussian process $\{B_{t}\}_{0\leq t\leq T}$ is called fractional Brownian motion (abbreviated fBM) of Hurst parameter $H\in(0,1)$ if its covariance function $R(s,t) = \mathrm{E}\{B(t)B(s)\},\ s,t\in[0,T]$ satisfies 
\begin{align}\label{def3}
R(t,s)  = \frac{1}{2}(t^{2H}+s^{2H} - |t-s|^{2H})
\end{align}
for $s,t \in [0,T]$.
\end{definition}

\begin{remark}
In the sequel, we always assume that the Hurst parameter $H$ is greater than $1/2$.
\end{remark}

When we consider lead-lag relationship between time series, it is reasonable to assume that time series are dependent on each other in appropriate sense. In order to introduce a dependent structure between fBMs, we first define correlated BMs and then construct fBMs from them. Therefore we need a representation of fBM via BM.

\begin{proposition}
Let $ W = (W_{t})_{t\in[0,T]} $ be a standard BM. Consider the square integrable kernel
\begin{align*}
K_{H}(t,s) = c_{H}\mathbf{1}_{(0,t)}(s)s^{1/2-H}\int_{s}^{t} (u-s)^{H-3/2}u^{H-1/2}\ du,\ 0\leq t\leq T
\end{align*}
where $ c_{H} = \sqrt{ \frac{H(2H-1)}{B(2-2H,H-1/2)} } $. Then the process $ B=(B_{t})_{t\in[0,T]} $ defined by
\begin{align}\label{def2}
B_{t} = \int_{0}^{T} K_{H}(t,s)\ dW_{s}
\end{align}
is an fBM of Hurst parameter $H$. Here the integral in (\ref{def2}) is interpreted as a Wiener integral.
\end{proposition}
\begin{proof}
Since the process $B$ is clearly a centered Gaussian process, it suffices to show (\ref{def3}). Later we calculate $R(t,s)$ in a more general setting and (\ref{def3}) is obtained as a corollary, see Proposition \ref{fBM2} and Corollary \ref{corollary3}.
\end{proof}

Let $\mathcal{E}$ be the set of step functions on $[0,T]$. Let us consider the linear operator $ K_{H}^{\ast}\colon\mathcal{E}\to L^{2}([0,T]) $ defined by
\begin{align*}
(K_{H}^{\ast}\varphi)(s) = \int_{s}^{T} \varphi(t)\frac{\partial}{\partial{t}}K_{H}(t,s)\ dt. 
\end{align*}
Note that 
\begin{align*}
(K_{H}^{\ast}\mathbf{1}_{[0,t]})(s) = K_{H}(t,s)\mathbf{1}_{[0,t]}(s).
\end{align*}
Let $ \mathcal{H}_{H} $ be the completion of $\mathcal{E}$, with respect to the inner product
\begin{align*}
\nonumber\langle f,g \rangle_{\mathcal{H}_{H}} &= H(2H-1) \int_{0}^{T}dr\int_{0}^{T}du\ f(r)g(u)|r-u|^{2H-2} \\
&= \langle K_{H}^{\ast}f , K_{H}^{\ast}g \rangle_{L^{2}([0,T])}.
\end{align*}
The operator $K_{H}^{\ast}$ can be extended to an isometry between $\mathcal{H}_{H}$ and $L^{2}([0,T])$, and the extension is also denoted by $K_{H}^{\ast}$. We set $ \| f \|^{2}_{\mathcal{H}} = \langle f,f \rangle_{\mathcal{H}} $ for $ f\in\mathcal{H} $. The space $\mathcal{H}_{H}$ is isomorphic to a Gaussian Hilbert space
\begin{align*}
\mathfrak{H}_{H} := \overline{ \mathrm{span}\left\{ B_{t} ; t\in [0,T] \right\} }^{L^{2}(\mathbb{P})}.
\end{align*}
We denote the isomorphism between $ \mathcal{H}_{H} $ and $ \mathfrak{H}_{H} $ by $ B $. For proofs of these results, see \cite{nualart2006malliavin}.

In order to prove Theorem \ref{theorem1}, we need the notion of multiple Wiener integral and its properties. Let us collect some basic facts concerning multiple Wiener integral which will be used in subsequent sections of this paper.

\begin{definition}
Let $\mathcal{H}$ be a real separable Hilbert space. A centered Gaussian family of random variables $(W(h))_{h\in\mathcal{H}}$ indexed by $\mathcal{H}$ is said to be an isonormal Gaussian process over $\mathcal{H}$ if the covariance function $ \mathrm{E}\{W(g)W(h)\} = \langle g,h  \rangle_{\mathcal{H}} $ for $g,h\in\mathcal{H}$.
\end{definition}

\begin{example}\label{BM1} 
Let $\mathcal{H} = L^{2}([0,T];\mathbb{R}^{2})$ with an inner product 
\begin{align*}
 \langle (f_{1},g_{1}) , (f_{2},g_{2}) \rangle_{\mathcal{H}} = \langle f_{1},g_{1} \rangle_{L^{2}([0,T])} + \langle f_{2},g_{2} \rangle_{L^{2}([0,T])}.
\end{align*}
Let $ \tilde{W} = ( \tilde{W}^{1}, \tilde{W}^{2} ) $ be a two-dimensional BM. Then 
\begin{align*}
W((f,g)) := \int_{0}^{T} f(t)\ d\tilde{W}_{t}^{1} + \int_{0}^{T} g(t)\ d\tilde{W}^{2}_{t}
\end{align*}
is an isonormal Gaussian process over $\mathcal{H}$. 
\end{example}

Let $ p $ be a positive integer and $ h \in \mathcal{H}^{\tilde{\otimes}p} $, where $\mathcal{H}^{\tilde{\otimes}p}$ is the $p$th symmetric tensor power of $\mathcal{H}$. After defining an isonormal Gaussian process over $\mathcal{H}$, we can consider the $p$th multiple Wiener integral $\mathbb{I}_{p}(h)$. Note that $ \mathbb{I}_{1}(f) $ coincides with $W(f)$ for $f\in\mathcal{H}$. A detailed discussion about multiple Wiener integral can be found in, for example, Chapter 2 of \cite{nourdin2012normal}. The following properties of multiple Wiener integral are useful. 

\begin{proposition}[Isometry property]\label{proposition3}
Let $p $ and $q$ be positive integers with $1\leq p\leq q$. For $ f \in \mathcal{H}^{\tilde{\otimes}p} $ and $ g \in \mathcal{H}^{\tilde{\otimes}p} $, we have 
\[
\mathrm{E}\{ \mathbb{I}_{p}(f) \mathbb{I}_{q}(g) \} = \begin{cases} 
p!\langle f,g \rangle_{\mathcal{H}^{\otimes p}} & \text{if}\ p=q, \\
0 & \text{else}.
\end{cases}
\] 
\end{proposition}
\begin{proof}
For example, see Proposition 2.7.5 of \cite{nourdin2012normal}.
\end{proof}

\begin{theorem}[Product formula]\label{theorem2}
Let $p,q \geq 1$ be positive integers. For $ f \in \mathcal{H}^{\tilde{\otimes}p} $ and $ g \in \mathcal{H}^{\tilde{\otimes}p} $, we have 
\[
\mathbb{I}_{p}(f)\mathbb{I}_{q}(g) = \sum_{r=0}^{p\wedge q} r! \left( \begin{array}{cc}p\\ r\end{array} \right)\left( \begin{array}{cc}q\\ r\end{array} \right)\mathbb{I}_{p+q-2r} (f \tilde{\otimes}_{r} g),
\]
where $f \tilde{\otimes}_{r} g$ denotes the symmetrization of the $r$th contraction of $f$ and $g$ (for a detail, see Appendix B of \cite{nourdin2012normal}). In particular, if $p=q=1$, then
\[
\mathbb{I}_{1}(f)\mathbb{I}_{1}(g) = \langle f,g \rangle_{\mathcal{H}} + \mathbb{I}_{2}(f \tilde{\otimes} g),
\]
where $f\tilde{\otimes} g = (1/2)(f\otimes g + g \otimes f)$.
\end{theorem}
\begin{proof}
For example, see Theorem 2.7.10 of \cite{nourdin2012normal}.
\end{proof}

\begin{theorem}[Hypercontractivity]\label{theorem3}
Let $q>0$ be a positive number, $p$ be a positive integer, and $h \in \mathcal{H}^{\tilde{\otimes}p}$. Then there exists a constant $C(p,q)\in(0,\infty)$ depending only on $p$ and $q$ such that
\[
\mathrm{E}\{ | \mathbb{I}_{p}(h) |^{q} \}^{1/q} \leq C(p,q) \mathrm{E}\{ | \mathbb{I}_{p}(h) |^{2} \}^{1/2}.
\]
\end{theorem}
\begin{proof}
For example, see Theorem 2.7.2 of \cite{nourdin2012normal}.
\end{proof}

\subsection{Correlated fBMs}\label{section2}

Let $(\Omega,\mathcal{F},\mathbb{P})$ be a stochastic basis supporting a two dimensional standard Brownian motion $ W=(\tilde{W}^{1},\tilde{W}^{2}) $. Here we assume that the $\sigma$-field $ \mathcal{F} $ is complete and the filtration $\mathbb{G}$ satisfies the usual conditions. For $\rho\in[-1,1]$, we set 
\begin{align*}
\begin{cases}
W^{1}_{t} = \tilde{W}_{t}^{1} , \\
W^{2}_{t} = \rho \tilde{W}_{t}^{1} + \sqrt{1-\rho^{2}} \tilde{W}_{t}^{2},
\end{cases}
\end{align*}
for $t\geq0$. Then the processes $W^{1}$ and $W^{2}$ are Brownian motions satisfying
\begin{align*}
\langle W^{1}, W^{2} \rangle_{t} = \rho t,\ t\geq0.
\end{align*}
We define correlated fractional Brownian motions on an interval $[0,T]$ by
\begin{align}\label{fBM1} 
\begin{cases}
\tilde{B}^{1}_{t} = \int_{0}^{T} K_{H_{1}}(t,s)\ dW^{1}_{s}, \\
\tilde{B}^{2}_{t} = \int_{0}^{T} K_{H_{2}}(t,s)\ dW^{2}_{s},
\end{cases}
\end{align}
for $H_{1},H_{2} \in (1/2,1)$. Here the integrals in (\ref{fBM1}) are understood in the Wiener sense. In terms of an isonormal Gaussian process, we can write (\ref{fBM1}) as 
\begin{align} \label{fBM3}
\begin{cases}
\tilde{B}^{1}_{t} = W( ( K_{H_{1}}^{\ast}\mathbf{1}_{[0,t]} , 0 ) ), \\
\tilde{B}^{2}_{t} = W( ( \rho K_{H_{2}}^{\ast}\mathbf{1}_{[0,t]} , \sqrt{1-\rho^{2}} K_{H_{2}}^{\ast}\mathbf{1}_{[0,t]} ) )
\end{cases}
\end{align}
(see Example \ref{BM1}). For given subinterval $ I \subset [0,T] $, we set 
\begin{align}\label{def1} 
\begin{cases}
h^{1}(I) = ( K_{H_{1}}^{\ast}\mathbf{1}_{I} , 0 ) , \\
h^{2}(I) = ( \rho K_{H_{2}}^{\ast}\mathbf{1}_{I} , \sqrt{1-\rho^{2}} K_{H_{2}}^{\ast}\mathbf{1}_{I} ) .
\end{cases}
\end{align}
Note that since the kernel $ K_{H_{i}}(t,s) $ is a Volterra kernel, i.e., $K_{H_{i}}(t,s)=0$ if $s\in[t,T]$, the process $\tilde{B}^{i}$ is $\mathbb{G}$-adapted. We can take a modification of the process $ \tilde{B}^{i} $ which has $ H_{i}-\epsilon $-H\"{o}lder continuous trajectory for each $\epsilon\in(0,H_{i})$ and is denoted by $ {B}^{i} $. The covariance between $B^{1}$ and $B^{2}$ are given by (\ref{calc4}) (see Proposition \ref{fBM2} below).

\begin{proposition} \label{fBM2} 
Let $ B^{1}_{t}=(B^{1}_{t})_{t\in[0,T]} $ and $ B^{2} = (B^{2}_{t})_{t\in[0,T]} $ be fractional Brownian motions on $[0,T]$ defined as above. Then we have
\begin{align} \label{calc4} 
\mathbb{E}\{ B_{t}^{1}B_{s}^{2} \} = \rho c_{H_{1}}c_{H_{2}} \int_{0}^{t}du\int_{0}^{s}dv\ \beta(u,v)|u-v|^{H_{1}+H_{2}-2} u^{H_{1}-H_{2}} v^{H_{2}-H_{1}},
\end{align}
where
\begin{align*}
\beta(u,v)= 
\begin{cases}
B(H_{1}-1/2, 2-H_{1}-H_{2}) & \text{if}\ u\leq v, \\
B(H_{2}-1/2,2-H_{1}-H_{2}) & \text{if}\ v < u.
\end{cases}
\end{align*}
Here $B(a,b)$ denotes the Beta function.
\end{proposition}
\begin{proof}
We have
\begin{align}\label{calc1} 
\nonumber \mathbb{E}\{ B_{t}^{1} B_{s}^{2} \} &= \rho\int_{0}^{T} K_{H_{1}}(t,r)K_{H_{2}}(s,r)\ dr \\
\nonumber &= \rho c_{H_{1}} c_{H_{2}} \int_{0}^{t\wedge s}dr\ r^{1-H_{1}-H_{2}} \\ 
\nonumber & \quad \times\int_{r}^{t}du\ (u-r)^{H_{1}-3/2}u^{H_{1}-1/2} \int_{r}^{s}dv\ (v-r)^{H_{2}-2/3} v^{H_{2}-1/2} \\
\nonumber &= \rho c_{H_{1}} c_{H_{2}} \int_{0}^{t}du\ \int_{0}^{s}dv\ u^{H_{1}-1/2}v^{H_{2}-1/2}
\\& \quad \times \int_{0}^{u\wedge v}dr\ r^{1-H_{1}-H_{2}} (u-r)^{H_{1}-3/2} (v-r)^{H_{2}-3/2}.
\end{align}
Suppose that $ u\leq v $. By the change of variables $ z=\frac{v-r}{u-r} $ and $ x=\frac{v}{uz} $, we obtain 
\begin{align}\label{calc2} 
\nonumber & \int_{0}^{u}dr\ r^{1-H_{1}-H_{2}} (u-r)^{H_{1}-3/2} (v-r)^{H_{2}-3/2} \\ 
\nonumber &= (v-u)^{H_{1}+H_{2}-2}\int_{v/u}^{\infty}dz\ (zu-v)^{1-H_{1}-H_{2}} z^{H_{2}-3/2} \\
&= (v-u)^{H_{1}+H_{2}-2}u^{1/2-H_{2}}v^{1/2-H_{1}} B(H_{1}-1/2,2-H_{1}-H_{2}).
\end{align}
The case where $ v<u $ can be obtained similarly:
\begin{align}\label{calc3} 
\nonumber \int_{0}^{v}dr\ r^{1-H_{1}-H_{2}} (u-r)^{H_{1}-3/2} (v-r)^{H_{2}-3/2} &= (u-v)^{H_{1}+H_{2}-2}u^{1/2-H_{2}}v^{1/2-H_{1}} \\
& \quad \times B(H_{2}-1/2,2-H_{1}-H_{2}).
\end{align}
Plugging (\ref{calc1}) and (\ref{calc2}) into (\ref{calc3}), we complete the proof.
\end{proof}

\begin{corollary}\label{corollary3}
Assume that $H_{1}=H_{2}=H$ and $\rho=1$. Then we have 
\begin{align*}
\nonumber \mathbb{E}\{ {B}_{t}^{1} {B}_{s}^{1} \} &= c_{H}^{2} B(H-1/2,2-2H) \int_{0}^{t}du\int_{0}^{s}dv\ |u-v|^{2H-2} \\
&= \frac{1}{2}(t^{2H} + s^{2H} - |t-s|^{2H}).
\end{align*}
\end{corollary}

We will use the representation (\ref{fBM3}) in the proof of Theorem \ref{theorem1}. The next proposition shows that we can always construct an isonormal Gaussian process $W$ over $L^{2}([0,T];\mathbb{R}^{2})$ satisfying (\ref{fBM3}) if we start with fBMs satisfying (\ref{calc4}).

\begin{proposition}
Let $ B^{1} =(B^{1}_{t})_{t\in[0,T]} $ and $ B^{2} = (B^{2}_{t})_{t\in[0,T]} $ be fBMs satisfying (\ref{calc4}). Then there exists an isonormal Gaussian process $W$ over $ L^{2}([0,T];\mathbb{R}^{2}) $ such that 
\begin{align} \label{fBM4}
\left\{\begin{array}{ll} 
{B}^{1}_{t} = W( ( K_{H_{1}}^{\ast}\mathbf{1}_{[0,t]} , 0 ) ), \\
{B}^{2}_{t} = W( ( \rho K_{H_{2}}^{\ast}\mathbf{1}_{[0,t]} , \sqrt{1-\rho^{2}} K_{H_{2}}^{\ast}\mathbf{1}_{[0,t]} ) )
\end{array}\right.
\end{align}
holds.
\end{proposition}
\begin{proof}
We define an isonormal Gaussian process over $ L^{2}([0,T]) $ by $ W^{l}(\varphi) = B^{l}( (K_{H_{l}}^{\ast})^{-1} \varphi ) $ for $ \varphi \in L^{2}([0,T]) $ and $l=1,2$. Then $ B_{t}^{l} = W^{l}(K_{H_{l}}^{\ast}\mathbf{1}_{[0,t]} ) $ by definition. For step functions $ f_{1} ,f_{2} $, it holds that
\begin{align}\label{calc11}
\mathbb{E}\{ B^{1}(f_{1})B^{2}(f_{2}) \} = \rho \langle K_{H_{1}}^{\ast}f_{1} , K_{H_{2}}^{\ast}f_{2} \rangle_{L^{2}([0,T])}.
\end{align}
We can verify (\ref{calc11}) holds for any $ f_{1} \in \mathcal{H}_{H_{1}} $ and $ f_{2} \in \mathcal{H}_{H_{2}} $ by approximating $ f_{l} $ (with respect to the norm $\| \cdot \|_{\mathcal{H}_{H_{l}}} $) by step functions. Therefore we have 
\begin{align}\label{calc12}
\mathbb{E}\{ W^{1}(\varphi)W^{2}(\psi) \} = \rho \langle \varphi, \psi \rangle_{L^{2}([0,T])}
\end{align}
for $ \phi,\psi \in L^{2}([0,T]) $. We set 
\begin{align*}
\begin{cases}
\tilde{W}^{1}(\varphi) = W^{1}(\varphi) \\
\tilde{W}^{2}(\psi) = (W^{2}(\psi) - \rho W^{1}(\psi))/\sqrt{1-\rho^{2}} 
\end{cases}
\end{align*}
and 
\begin{align} \label{calc13}
 W( (\varphi,\psi) ) = \tilde{W}^{1}(\varphi) + \tilde{W}^{2}(\psi)  
\end{align}
for $ \phi,\psi \in L^{2}([0,T]) $. It is easy to check that the isonormal Gaussian process $W$ defined by (\ref{calc13}) satisfies (\ref{fBM4}).
\end{proof}

\section{Model assumptions}

Now we describe our model in this section. We give precise assumptions on the observed processes and the observation times. 

\subsection{Lead-lag between stochastic processes involving fBM}

Using the correlated fBMs introduced in Section \ref{section2}, we consider the following model for a lead-lag relationship between stochastic processes involving fBM.

\begin{assumption}\label{ass1}
Let $(\Omega,\mathcal{F},\mathbb{P})$ be an underlying probability space. If $\theta\in\Theta_{\geq 0}:=[0,\delta)$, then we assume that the two-dimensional process $X=(X_{t})_{t\in[0,T+\delta]}=( (X^{1}_{t}, X^{2}_{t}) )_{t\in[0,T+\delta]}$ is given by 
\begin{align*}
\begin{cases}
X_{t}^{1} = X_{0}^{1} + \sigma_{1} B_{t+\theta}^{1} + A^{1}_{t}, \\
X_{t}^{2} = X_{0}^{2} + \sigma_{2} B_{t}^{2} + A^{2}_{t},
\end{cases}
\end{align*}
for $t\in[0,T+\delta]$, where 
\begin{enumerate}
\item[(A1)] $X^{1}_{0}$ and $X^{2}_{0}$ are real-valued random variables, 
\item[(A2)] the processes $B^{1}$ and $B^{2}$ are fBMs satisfying (\ref{calc4}), and
\item[(A3)] the drift processes $A^{1}$ and $A^{2}$ are Lipschitz continuous $ \mathbb{P} $-almost surely. 
\end{enumerate}
Moreover, if $\theta\in\Theta_{< 0}:=(-\delta,0)$, then we assume that the process $ X^{\star} = (X^{2},X^{1}) $ has a representation 
\begin{align*}
\begin{cases}
X_{t}^{2} = X_{0}^{2} + \sigma_{2} B_{t-\theta}^2 + A^{2}_{t}, \\
X_{t}^{1} = X_{0}^{1} + \sigma_{1} B_{t}^{1} + A^{1}_{t},
\end{cases}
\end{align*}
for $t\in[0,T+\delta]$ where the conditions (A1)-(A3) are in force.
\end{assumption}

A lead-lag fractional Black-Scholes model, in which we are interested, can be described as follows.

\begin{example} \label{fBS1} 

Let $(\Omega,\mathcal{F},\mathbb{P})$ be a probability space which supports fBMs $B^{1}$ and $B^{2}$ satisfying (\ref{calc4}). Here we assume that the $\sigma$-field $ \mathcal{F} $ is complete. 

Let us consider a system of stochastic differential equations
\begin{align}\label{fBS3}
\begin{cases}
S^{1}_{t} = S_{0}^{1} + \mu^{1} \int_{0}^{t} S^{1}_{s}\ ds + \sigma^{1} \int_{0}^{t} S^{1}_{s}\ dB_{s}^{1}, \\
S^{2}_{t} = S_{0}^{2} + \mu^{2} \int_{0}^{t} S^{2}_{s}\ ds + \sigma^{2} \int_{0}^{t} S^{2}_{s}\ dB_{s}^{2}, \\
\end{cases}
\end{align}
where $ t\in[T+2\delta] $, $ S_{0}^{i} >0 $, $ \mu^{i}\in\mathbb{R} $ and $ \sigma^{i} \in \mathbb{R}\setminus \{0\} $ ($i=1,2$). Then, as we noted in Section \ref{section1}, the solution of system (\ref{fBS3}) can be written as
\begin{align*}
S^{i}_{t} = S_{0} \exp( A_{t}^{i} + \sigma^{i} B_{t}^{i} ),
\end{align*}
where $A_{t}^{i} = \mu^{i}t$ if we use the Riemann-Stieltjes integral, and $ A_{t}^{i} = \mu^{i}t - \frac{(\sigma^{i})^{2}}{2}t^{2H} $ if we consider the Wick-It\^{o}-Skorokhod integral. Note that the function $A^{i}$ is Lipschitz continuous in either case.

Let $\theta\in[0,\delta)$. We define ${X}=({X}^{1}, {X}^{2})$ by
\begin{align*}
\begin{cases}
{X}_{t}^{1} = \log S_{t+\theta}^{1} = \log {S}_{0}^{1} + \sigma^{1} {B}_{t+\theta}^{1} + A_{t+\theta}^{1}, \\
{X}_{t}^{2} = \log S_{t}^{2} =\log {S}_{0}^{2} + \sigma^{2} {B}_{t}^{2} + A_{t}^{2},
\end{cases}
\end{align*}
for $t\in[0,T+\delta]$. Then the processes $X^{1}$ and $X^{2}$ satisfy Assumption \ref{ass1}.
\end{example} 

We give another example of the process that satisfies Assumption \ref{ass1} in Section \ref{section3}.

\subsection{Observation}

Now we give the assumptions on the observation. We consider the problem of estimating the lead-lag parameter $\theta\in\Theta$ from discretely observed $X^{1}$ and $X^{2}$. Since we keep in mind that the processes $X^{1}$ and $X^{2}$ are prices of stocks traded at high frequency, it is natural to consider asymptotics where the number of observations tend to infinity as $n\to\infty$. Therefore we assume that the process $X^{1}$ is observed at  
\[
\mathcal{T}^{1,n} = \{ 0 = t_{0}^{1,n} < t_{1}^{1,n} < \ldots < t_{N_{1}(n)}^{1,n} = T+\delta \},
\]
and the process $X^{2}$ at
\[
\mathcal{T}^{2,n} = \{ 0 = t_{0}^{2,n} < t_{1}^{2,n} < \ldots < t_{N_{2}(n)}^{2,n} = T+\delta \},
\]
where $n$, $N_{1}(n)$, and $N_{2}(n)$ are positive integers. Note that they are in general unevenly spaced, non-synchronous and may be random and depend on $X$. Now we introduce some notations that will be useful for describing the assumptions on the observation times $ \mathcal{T}^{1,n} $ and $ \mathcal{T}^{2,n} $. 


%


\begin{notation}
We set
\begin{itemize}
\item $ I_{i}^{k,n} = (t_{i-1}^{k,n}, 1_{i}^{k,n}],\ \ i=1,\ldots,N_{k}(n)$ for $k=1,2$, 
\item $ | I_{i}^{k,n} | = t_{i}^{k,n}-t_{i-1}^{k,n},\ i=1,\ldots,N_{k}(n)$ for $k=1,2$,  
\item $ \mathcal{I}^{k,n} = \{ I_{i}^{k,n} \mid i=1,\ldots,N_{k}(n) \} $ for $k=1,2$,
\item $ r_{n} = \max\{ |I_{i}^{1,n}|,|I_{j}^{2,n}|\mid i=1,\ldots,N_{1}(n),\ j=1,\ldots,N_{2}(n) \} $ (the maximum mesh size of the observation times at the step $n$),
\item $ \underline{I} = \inf \{ x\mid x\in I \} $ for an interval $I$,
\item $ \overline{I}  = \sup\{ x \mid x\in I \}$ for an interval $I$,
\item $I_{\theta}$ = $\{ x+\theta \mid x\in I \}$ for an interval $I$ and $\theta\in\mathbb{R}$, 
\item $\mathcal{I}^{k,n}_{\theta} = \{ (I_{i}^{k,n})_{\theta} \mid i=1,\ldots,N_{k}(n) \}$ for $k=1,2$ and $\theta\in\mathbb{R}$,
\item $\mathcal{K}(I) = \bigcup_{ k\colon K_{k}\cap I \neq \emptyset } K_{k}$ for given subintervals $\mathcal{K}:=\{ K_{k} \mid k \}$ and $I$ (if there are no $ k $ such that $ K_{k} \cap I \neq \emptyset $, then we set $ \mathcal{K}(I) = \emptyset $),
\item $ f(I) = f(\overline{I}) - f(\underline{I}) $ for a function $ f $ and an interval $ I $, and
\item $V_{\geq0} = V\cap [0,\infty)$ and $V_{<0} = V\cap (-\infty,0)$ for a subset $V$ of $\mathbb{R}$.
\end{itemize}
\end{notation}

\begin{notation}
Let $\{a_{i}\}_{i\in I}$ be a real sequence indexed by $I\subset\mathbb{Z}_{\geq0}$. If we omit the range of summation, then we mean the summation over $I$, that is, $\sum_{i}a_{i} = \sum_{i\in I}a_{i}$. If the index $I$ is empty, then we set $\sum_{i\in I} a_{i} = 0$. We also denote $\bigcup_{i\in I}A_{i} $ (resp. $\bigcap_{i\in I} A_{i}$) as $ \bigcup_{i}A_{i} $ (resp. $\bigcap_{i}A_{i}$) for sets $\{A_{i}\}_{i\in I}$.
\end{notation}

Now we give precise assumptions on the observation times. We give some examples that satisfy Assumption \ref{ass2} in Section \ref{section4}.

\begin{assumption}\label{ass2} 
Let $(\Omega,\mathcal{F},\mathbb{P})$ be an underlying probability space. Suppose that there are random variables $\{N_{l}(n) \in \mathbb{N} \mid l=1,2;n\in\mathbb{N}\}$ and $\{t_{i}^{l,n}\mid i=0,1\ldots,N_{l}(n);l=1,2;n\in\mathbb{N}\}$ such that, for $l=1,2$ and $n\in\mathbb{N}$, 
\[
0=t_{0}^{l,n} < t_{1}^{l,n} < \ldots <  t_{N_{l}(n)}^{l,n} = T+\delta,
\]
hold $\mathbb{P}$-almost surely. We also assume that $\bigcap_{l=1,2}\{i\mid \overline{I_{i}^{l,n}} \leq T \} \neq \emptyset$ for all $n\in\mathbb{N}$. We set $\mathcal{T}^{l,n} = \{ t_{i}^{l,n} \mid i=1,\ldots,N_{l}(n) \} $. 

The observation times $\mathcal{T}^{1,n}$ and $\mathcal{T}^{2,n}$ satisfy following conditions.
\begin{enumerate}
\item[(B1)] The $\sigma$-algebra $ \sigma(\cup_{n\in\mathbb{N}}\mathcal{T}^{1,n}) $ generated by the observation times $ \mathcal{T}^{1,n},\ n\in\mathbb{N} $ is independent of the $\sigma$-algebra $ \sigma(\cup_{n\in\mathbb{N}}\mathcal{T}^{2,n}) $ generated by $ \mathcal{T}_{2,n},\ n\in \mathbb{N} $. Moreover the $\sigma$-algebra $ \sigma(\mathcal{T}) := \sigma(\cup_{n\in\mathbb{N}}(\mathcal{T}^{1,n} \cup \mathcal{T}^{2,n})) $ is independent of the fBMs $ B^{1} $ and $B^{2}$. 
\item[(B2)] There exist positive constants $ c_{\ast}>0 $ and $\epsilon>0$ such that 
\begin{align} \label{limit7}
\mathbb{P}\left\{ \frac{ \sum_{i:\overline{I_{i}^{1,n}} \leq T, \underline{I_{i}^{1,n}} \geq \epsilon} | I_{i}^{1,n} |^{H_{1}+H_{2}} }{ \sqrt{\sum_{i} |I_{i}^{1,n}|^{2H_{1}} } \sqrt{ \sum_{j} |I_{j}^{2,n}|^{2H_{2}} } }  \geq c_{\ast} \right\} \to 1
\end{align}
and
\begin{align} \label{limit8}
\mathbb{P}\left\{ \frac{ \sum_{j:\overline{I_{j}^{2,n}} \leq T, \underline{I_{j}^{2,n}} \geq \epsilon} | I_{i}^{2,n} |^{H_{1}+H_{2}} }{ \sqrt{\sum_{i} |I_{i}^{1,n}|^{2H_{1}} } \sqrt{ \sum_{j} |I_{j}^{2,n}|^{2H_{2}} } }  \geq c_{\ast} \right\} \to 1
\end{align}
hold as $n\to\infty$.
\end{enumerate}
Moreover, we assume that there exists a deterministic sequence $ \{v_{n}\}_{n\in\mathbb{N}} $ such that $ v_{n} \in (0,\delta) $ for all $n\in\mathbb{N}$, $ v_{n} \to 0 $ as $ n\to\infty $ and the following properties (B3) and (B4) hold.
\begin{enumerate}
\item[(B3)] For any $\mu>0$, there exists $ \gamma(\mu) > \mu $ such that
\[
\frac{ r_{n}^{2H_{1}-1+\mu} }{ \sum_{i:\overline{I_{i}^{1,n}} \leq T } | I_{i}^{1,n} |^{2H_{1}} }  = O_{p}( v_{n}^{\gamma(\mu)} )
\]
and
\[
\frac{ r_{n}^{2H_{2}-1+\mu} }{ \sum_{j: \overline{I_{j}^{2,n}} \leq T} | I_{j}^{2,n} |^{2H_{2}} }  = O_{p}( v_{n}^{\gamma(\mu)} )
\]
hold.
\item[(B4)] It holds that
\begin{align*}
 \lim_{n\to\infty} v_{n}^{2-(H_{1}\vee H_{2})} (\mathbb{E} \{ \#\mathcal{I}^{1,n} \} + \mathbb{E} \{ \# \mathcal{I}^{2,n} \}) = 0. 
\end{align*}
\end{enumerate}
\end{assumption}

\begin{remark}
The assumption $\bigcap_{l=1,2}\{i\mid \overline{I_{i}^{l,n}} \leq T \} \neq \emptyset$ for all $n\in\mathbb{N}$ is in fact not a constraint. Since 
\[
\bigcup_{l=1,2} \bigcap_{i} \left\{ \overline{ I_{i}^{l,n} } > T \right\} \subset \bigcup_{l=1,2} \left\{ \frac{ \sum_{i:\overline{I_{i}^{l,n}} \leq T, \underline{I_{i}^{l,n}} \geq \epsilon} | I_{i}^{l,n} |^{H_{1}+H_{2}} }{ \sqrt{\sum_{i} |I_{i}^{1,n}|^{2H_{1}} } \sqrt{ \sum_{j} |I_{j}^{2,n}|^{2H_{2}} } } =0 \right\}
\]
holds, we have $\mathbb{P}\{\bigcap_{l=1,2} \bigcup_{i} \{ \overline{ I_{i}^{l,n} } \leq T\}\} \to 1 $ as $n\to\infty$ by (B2). Hence we can assume $\bigcap_{l=1,2}\{i\mid \overline{I_{i}^{l,n}} \leq T \} \neq \emptyset$ outside an asymptotically negligible set.
\end{remark}

Assumption \ref{ass2} is for consistency of the lead-lag estimator defined later in Definition \ref{def4}. These conditions may seem strange in comparison with the assumptions in \cite{hoffmann2013estimation}. In Section \ref{sec5}, we compare Assumption \ref{ass2} with the assumptions in \cite{hoffmann2013estimation}. 

\section{Construction of an esimator for the lead-lag parameter} \label{sec5}

Let us now turn to construct an estimator for the lead-lag parameter $\theta\in\Theta$. First we explain the idea for our estimator. Then we define our estimator in Definition \ref{def4} and state the main theorem of this paper (Theorem \ref{theorem1}). Finally we compare our consistency result with that of \cite{hoffmann2013estimation}. In particular, we consider differences of the assumptions on the sampling scheme and the grid between our case and the semimartingale case.

\subsection{The idea for the lead-lag estimator} \label{sec7}

To explain the idea for our estimator, let us outline the proof of consistency of the lead-lag estimator in semimartingale case. For now, we assume that the observed process $X=(X^{1},X^{2})$ is a regular semimartingale with lead-lag parameter $\theta\in[0,\delta)$. We also assume that the drift terms $A^{1}$ and $A^{2}$ are zero for simplicity. Let $\tilde{\theta}_{n}\in[0,\delta) \cap \mathcal{G}^{n}$. 

Suppose first that $| \tilde{\theta}_{n}-\theta |$ converges to $0$ rapidly as $n\to\infty$.
Then the Hayashi-Yoshida covariance contrast $\mathcal{U}_{n}(\tilde{\theta}_{n})$ is ``nearly'' the Hayashi-Yoshida estimator:
\begin{align*}
\mathcal{U}_{n}(\tilde{\theta}_{n}) &= \sum_{i,j: \overline{I}_{i}^{1,n}\leq T}X^{1}(I_{i}^{1,n})X^{2}(I_{j}^{2,n})\mathbf{1}_{I_{i}^{1,n}\cap(I_{j}^{2,n})_{-\tilde{\theta}_{n}}\neq \emptyset} \\
&= \sum_{i,j: \overline{I}_{i}^{1,n}\leq T} X^{1}(I_{i}^{1,n})\tau_{-\tilde{\theta}_{n}}(X^{2})( (I_{j}^{2,n})_{-\tilde{\theta}_{n}} ) \mathbf{1}_{I_{i}^{1,n}\cap(I_{j}^{2,n})_{-\tilde{\theta}_{n}}\neq \emptyset}.
\end{align*}
Therefore, the limit
\begin{align} \label{limit18}
 \mathcal{U}_{n}(\tilde{\theta}_{n}) \to^{p} \langle X, \tau_{-\theta}(Y) \rangle_{T} 
\end{align}
holds as $n\to\infty$. This is Proposition 4 of \cite{hoffmann2013estimation}. For the properties of the Hayashi-Yoshida estimator, we refer to \cite{hayashi2008consistent, hayashi2008asymptotic, hayashi2005covariance} for example. 

On the other hand, we can prove
\begin{align} \label{limit19}
\sup_{\tilde{\theta}\in \mathcal{G}^{n} \cap \{\tilde{\theta}\in \Theta \mid |\tilde{\theta}-\theta|\geq 2v_{n}\}} |\mathcal{U}_{n}(\tilde{\theta})| \to^{p} 0
\end{align}
as $n\to\infty$. This is a consequence of Proposition 3 of \cite{hoffmann2013estimation}. If $|\tilde{\theta} - \theta| \geq 2v_{n}$, then the contrast function $\mathcal{U}_{n}(\tilde{\theta})$ can be rewritten as the terminal value of a martingale outside of an asymptotically negligible set. To handle the probability of the supremum, the B\"{u}rkholder-Davis-Gundy inequality is used in \cite{hoffmann2013estimation}. 

By (\ref{limit18}) and (\ref{limit19}), the contrast function $\mathcal{U}_{n}(\tilde{\theta})$ has a peak at $\tilde{\theta}=\theta$ asymptotically on the set $\{\langle X, \tau_{-\theta}(Y) \rangle_{T} \neq0\}$. Hence we can conclude that the maximum contrast estimator is consistent.

Now we go back to the original problem. Let us assume that the observed process $X$ satisfies Assumption \ref{ass1}. We essentially follow the strategy above. However, we can not use the Hayashi-Yoshida covariance contrast function directly. This is because the Hayashi-Yoshida estimator may converge to zero if the observed process is not a semimartingale. For example, let us set $A^{1}=A^{2}=0$, $\sigma_{1}=\sigma_{2}=0$, $H_{1}=H_{2}=H$ and $\rho=1$ in Assumption \ref{ass1} and $\mathcal{T}^{l,n}=\{ (1+\delta)i/n \mid i=1,\ldots,n \}$ for $l=1,2$ in Assumption \ref{ass2}. We also assume $\theta=0$ for simplicity. Then we obtain 
\begin{align*}
\mathcal{U}_{n}(0) &= \frac{1}{n^{2H-1}} n^{2H-1}\sum_{i=1}^{\lfloor n/(1+\delta) \rfloor} (B^{1}_{i/n}- B^{1}_{(i-1)/n} )^{2} \\
&\to^{p} 0
\end{align*}
as $n\to\infty$ by ergodicity and self-similarity of fractional Gaussian noise. To ensure that the contrast function is away from zero when $| \tilde{\theta}_{n}-\theta |$ converges to $0$ rapidly as $n\to\infty$, we consider the ``correlation version'' of the contrast $\mathcal{U}_{n}(\tilde{\theta})$. We consider the \textit{shifted Hayashi-Yoshida correlation contrast function} $\mathcal{U}^{\mathrm{cor}}_{n}\colon \Theta\to\mathbb{R}$ defined by
\begin{align*}
\nonumber \mathcal{U}^{\mathrm{cor}}_{n}(\tilde{\theta}) &= \mathbf{1}_{\Theta_{\geq0}}(\tilde{\theta}) \frac{ \sum_{i,j:\overline{I_{i}^{1,n}}\leq T} X^{1}(I_{i}^{1,n}) X^{2}(I_{j}^{2,n}) \mathbf{1}_{I_{i}^{1,n} \cap (I_{j}^{2,n})_{-\tilde{\theta}} \neq \emptyset } }{ \sqrt{ \sum_{i:\overline{I_{i}^{1,n}}\leq T} X^{1}(I_{i}^{1,n})^{2} } \sqrt{ \sum_{j} X^{2}(I_{j}^{2,n})^{2} } } \\
& \quad + \mathbf{1}_{\Theta_{<0}}(\tilde{\theta}) \frac{ \sum_{i,j:\overline{I_{j}^{2,n}}\leq T} X^{1}(I_{i}^{1,n}) X^{2}(I_{j}^{2,n}) \mathbf{1}_{ (I_{i}^{1,n})_{\tilde{\theta}} \cap I_{j}^{2,n} \neq \emptyset } }{ \sqrt{ \sum_{i} X^{1}(I_{i}^{1,n})^{2} } \sqrt{ \sum_{j:\overline{I_{j}^{2,n}}\leq T} X^{2}(I_{j}^{2,n})^{2} } }.
\end{align*}
Then we can ensure that the contrast $\mathcal{U}^{\mathrm{cor}}_{n}(\tilde{\theta}_{n})$ is away from zero when $| \tilde{\theta}_{n}-\theta |$ converges to $0$ rapidly as $n\to\infty$.

We have another problem. In the semimartingale case, the proof of (\ref{limit19}) relies on the martingale property of the observed process. However, we cannot expect it in our case. We overcome this problem by exploiting the Gaussian property of the observed process: the Wiener chaos decomposition and the hypercontractivity.

\subsection{Definition of the lead-lag estimator}
Following the discussion above, we now define the lead-lag estimator in our case. 

\begin{definition}\label{definition1}
We define the \textit{shifted Hayashi-Yoshida correlation contrast function} $\mathcal{U}^{\mathrm{cor}}_{n}\colon \Theta\to\mathbb{R}$ by
\begin{align}\label{def5}
\nonumber \mathcal{U}^{\mathrm{cor}}_{n}(\tilde{\theta}) &= \mathbf{1}_{\Theta_{\geq0}}(\tilde{\theta}) \frac{ \sum_{i,j:\overline{I_{i}^{1,n}}\leq T} X^{1}(I_{i}^{1,n}) X^{2}(I_{j}^{2,n}) \mathbf{1}_{I_{i}^{1,n} \cap (I_{j}^{2,n})_{-\tilde{\theta}} \neq \emptyset } }{ \sqrt{ \sum_{i:\overline{I_{i}^{1,n}}\leq T} X^{1}(I_{i}^{1,n})^{2} } \sqrt{ \sum_{j} X^{2}(I_{j}^{2,n})^{2} } } \\
& \quad + \mathbf{1}_{\Theta_{<0}}(\tilde{\theta}) \frac{ \sum_{i,j:\overline{I_{j}^{2,n}}\leq T} X^{1}(I_{i}^{1,n}) X^{2}(I_{j}^{2,n}) \mathbf{1}_{ (I_{i}^{1,n})_{\tilde{\theta}} \cap I_{j}^{2,n} \neq \emptyset } }{ \sqrt{ \sum_{i} X^{1}(I_{i}^{1,n})^{2} } \sqrt{ \sum_{j:\overline{I_{j}^{2,n}}\leq T} X^{2}(I_{j}^{2,n})^{2} } }.
\end{align}
if $\sum_{i:\overline{I_{i}^{l,n}}\leq T} X^{l}(I_{i}^{l,n})^{2} \neq 0$ for $l=1,2$. When $\sum_{i:\overline{I_{i}^{l,n}}\leq T} X^{l}(I_{i}^{l,n})^{2} = 0$ for some $l$, we set $\mathcal{U}^{\mathrm{cor}}_{n}(\tilde{\theta})=0$. Hereafter, we denote the contrast function $\mathcal{U}^{\mathrm{cor}}_{n}(\tilde{\theta})$ in (\ref{def5}) by $\mathcal{U}_{n}(\tilde{\theta})$ for notational simplicity. 
\end{definition}


Our estimator $ \hat{\theta}_{n} $ is obtained by maximizing the contrast function $ | \mathcal{U}_{n}(\tilde{\theta}) | $ over a finite grid $\mathcal{G}^{n}$ in the parameter space $\Theta$. The assumptions on the grid $\mathcal{G}^{n}$ are as follows.
\begin{assumption}\label{ass3}
Let $\mathcal{G}^{n}$ be a finite subset of the parameter space $\Theta$ such that 
\begin{enumerate}
\item[(C1)] $0\in\mathcal{G}^{n}$,
\item[(C2)] $\#\mathcal{G}^{n} = O(v_{n}^{-\beta})$ for some $\beta>0$, and
\item[(C3)] for some deterministic sequence $\rho_{n}>0$ and positive constant $\varsigma\in(0,1-(H_{1}\vee H_{2}))$, we have
\begin{align*}
\Theta \subset \bigcup_{\tilde{\theta}\in\mathcal{G}^{n}} [\tilde{\theta}-\rho_{n},\tilde{\theta}+\rho_{n}]
\end{align*}
and 
\[
\lim_{n\to\infty} \rho_{n}^{1-\varsigma}(\mathbb{E}\{\#\mathcal{I}^{1,n}\} + \mathbb{E}\{\#\mathcal{I}^{2,n}\}) = 0.
\]
\end{enumerate}
Here the sequence $ (v_{n})_{n\in\mathbb{N}} $ is from Assumption \ref{ass2}. 
\end{assumption}

\begin{example}
An example of the grid $\mathcal{G}^{n}$ satisfying Assumption \ref{ass3} is 
\begin{align*}
\mathcal{G}^{n} = \{ v_{n}^{2-(H_{1}\vee H_{2})+\epsilon} m \mid m\in\mathbb{Z} \}\cap\Theta
\end{align*}
for any $\epsilon>0$ by (B4).
\end{example}


Now we are ready to define our lead-lag estimator.

\begin{definition} \label{def4}
Let $\mathcal{U}_{n}(\hat{\theta}_{n})$ be the shifted Hayashi-Yoshida correlation contrast function defined in Definition \ref{definition1}. We define the lead-lag estimator $\hat{\theta}_{n}$ as a solution of 
\begin{align*}
|\mathcal{U}_{n}(\hat{\theta}_{n})| = \max_{\tilde{\theta}\in\mathcal{G}^{n}} |\mathcal{U}_{n}(\tilde{\theta})|.
\end{align*}
\end{definition}

\begin{remark}

We can assume that $\hat{\theta}_{n}$ is chosen to be measurable. For example, let $ \hat{\theta}^{\ast}_{n} $ denote the largest element in the set $\arg\max_{\tilde{\theta}\in\mathcal{G}^{n}}|\mathcal{U}_{n}(\tilde{\theta})|$. Then we have
\[
\{ \hat{\theta}_{n}^{\ast} \leq r \} = \left\{ \max_{\tilde{\theta}\in\mathcal{G}^{n}\cap(r,\delta)} |\mathcal{U}_{n}(\tilde{\theta})| < \max_{\tilde{\theta}\in\mathcal{G}^{n}\cap [0,r]} |\mathcal{U}_{n}(\tilde{\theta})| \right\}
\]
for any $r\in[0,\delta)$, and hence $ \hat{\theta}_{n}^{\ast} $ is measurable.

\end{remark}

Let us state our main theorem.

\begin{theorem}\label{theorem1}
Suppose that 
\begin{itemize}
\item[(a)] the processes $X^{1}$ and $X^{2}$ are defined as in Assumption \ref{ass1} with $\rho\neq0$,  
\item[(b)] the observation times $\mathcal{T}^{1,n}$ and $\mathcal{T}^{2,n}$ satisfy Assumption \ref{ass2}, and
\item[(c)] the grid $\mathcal{G}^{n}$ satisfies Assumption \ref{ass3}.
\end{itemize}
Then the estimatior $\hat{\theta}_{n}$ satisfies 
\begin{align}\label{limit13}
v_{n}^{-1}(\hat{\theta}_{n} -\theta) \to^{p} 0
\end{align}
as $n\to \infty$. Here the symbol $\to^{p}$ denotes convergence in probability. Note that the sequence $(v_{n})_{n\in\mathbb{N}}$ is from Assumption \ref{ass2}.
\end{theorem} 

\begin{remark}
\begin{asparaenum}
\item It is interesting to consider removing the assumption $H\in(1/2,1)$. By constructing correlated fBMs from correlated BMs as in Section \ref{section2}, we could consider a lead-lag model between two processes involving fBMs with any Hurst parameter. However, we use the representation (\ref{calc4}) of the covariance when we prove the consistency of the lead-lag estimator. 

\item As we shall see in Section \ref{section7}, the convergence rate $v_{n}$ essentially equals $n^{-1}$ when the observation times are synchronous and equispaced with $n$ points or given by the Poisson sampling scheme with frequency proportional to $n$. However, as mentioned in \cite{hoffmann2013estimation}, the convergence rate $v_{n}$ could be improved in our case. 
In fact, \cite{koike2016asymptotic} analyzed the lead-lag model based on Gaussian likelihood and suggested the possibility such that the convergence rate becomes $n^{-3/2}$ when the observed processes are correlated Brownian motions with the microstructure noise and the observation is synchronous and evenly spaced with $n$ points.

\item For financial high-frequency data, the true price process are considered to be contaminated by \textit{market microstructure noise}, although we do not take into account the presence of it in this paper. In the semimartingale case, it is suggested in \cite{hoffmann2013estimation} that using a contrast function based on the pre-averaged Hayashi-Yoshida estimator gives a consistent lead-lag estimator even if microstructure noise is present. However, it is unclear whether the same approach works or not in our non-semimartingale case. For the properties of the pre-averaged Hayashi-Yoshida estimator, see \cite{koike2014limit, koike2016quadratic} and references therein. 
\end{asparaenum}
\end{remark}

Before proceeding the proof of Theorem \ref{theorem1}, let us compare our consistency result with the one of \cite{hoffmann2013estimation}. 

The condition (B1) in Assumption \ref{ass2} is a counterpart of the condition [B2] of \cite{hoffmann2013estimation}. Although they allowed the observation times to be dependent on the process, we content ourselves with the assumption that the observation times are independent of the process. 

The condition (B2) has no counterpart in \cite{hoffmann2013estimation}. This is because (B2) is naturally satisfied in semimartingale case. Let us set $H_{1}=H_{2}=1/2$ in (B2). Then (B2) becomes as follows.
\begin{enumerate}
\item[(B2$^{\prime}$)] There exist positive constants $c_{\ast}>0$ and $\epsilon>0$ such that
\[
\mathbb{P}\Biggl\{ \sum_{i:\overline{I_{i}^{l,n}}\leq T, \underline{I_{i}^{l,n}}\geq \epsilon} |I_{i}^{l,n}| \geq c_{\ast} \Biggr\} \to^{p} 1
\]
holds as $n\to\infty$ for $l=1,2$. 
\end{enumerate}
Let us assume $r_{n}=o_{p}(1)$. Since we consider high-frequency asymptotics, this is a minimum requirement. Then (B2$^{\prime}$) is derived from $r_{n}=o_{p}(1)$ since the inequality
\[
\mathbb{P}\Biggl\{ \sum_{i:\overline{I_{i}^{l,n}}\leq T, \underline{I_{i}^{l,n}}\geq T/2} |I_{i}^{l,n}| \geq T/4 \Biggr\} \geq \mathbb{P}\{r_{n} < T/8\}
\]
holds. 

The condition (B3) is a counterpart of the condition [B1] of \cite{hoffmann2013estimation}, that is, $r_{n} = o_{p}(v_{n})$. If $H_{1}=H_{2}=1/2$, then the condition (B3) becomes as follows.
\begin{enumerate}
\item[(B3$^{\prime}$)] For any $\mu>0$ there exists $\gamma(\mu)>\mu$ such that $r_{n}^{\mu} = O_{p}(v_{n}^{\gamma(\mu)})$ holds as $n\to\infty$. 
\end{enumerate}
The condition (B3$^{\prime}$) clearly implies $r_{n}=o_{p}(v_{n})$. Hence (B3) is a bit stronger than the condition [B1] of \cite{hoffmann2013estimation}.

The condition (B4) also has no counterpart in the semimartingale case. Unfortunately, the author has no explanation for the condition (B4). We use the condition (B4) only when we show that the contrast $\mathcal{U}_{n}(\tilde{\theta}_{n})$ is away from zero if $| \tilde{\theta}_{n}-\theta |$ converges to $0$ rapidly as $n\to\infty$. In semimartingale case, the consistency of the Hayashi-Yoshida estimator is exploited here. 

Finally, Assumption \ref{ass3} is a counterpart of [B3] in \cite{hoffmann2013estimation}.

%

\section{Proof of Theorem \ref{theorem1}} \label{sec6}

In this section, we give the proof of Theorem \ref{theorem1}. Without loss of generality, we can assume that $\theta \in \Theta_{\geq0}$. By Assumption \ref{ass1}, the case where $\theta \in \Theta_{<0}$ is equivalent to the case of $ -\theta  \in \Theta_{\geq 0}$ with indices 1 and 2 interchanged. Therefore the proof when $ \theta \in \Theta_{\geq0} $ can be applied in the case of $\theta\in\Theta_{<0}$. First we prove Theorem \ref{theorem1} assuming that $A^{1}=A^{2}=0$, and then we remove this assumption. 

\begin{notation}
Let $\{a_{n}\}_{n\in\mathbb{N}}$ and $\{b_{n}\}_{n\in\mathbb{N}}$ be sequences of real numbers. If there exists a positive constant $c>0$ independent of $n$ such that $a_{n} \leq c b_{n}$ for all $n$, then we write $a_{n} \lesssim b_{n}$.

\end{notation}

\subsection{Proof of Theorem \ref{theorem1} when $A^{1} = A^{2} =0$}

Under the additional assumption $A^{1} = A^{2} =0$, the contrast function $ \mathcal{U}_{n}(\tilde{\theta}) $ becomes
\begin{align}\label{calc8}
\nonumber \mathcal{U}_{n}(\tilde{\theta}) &= \mathbf{1}_{\Theta_{\geq0}}(\tilde{\theta}) \frac{ \sum_{i,j:\overline{I_{i}^{1,n}}\leq T} B^{1}((I_{i}^{1,n})_{\theta}) B^{2}(I_{j}^{2,n}) \mathbf{1}_{I_{i}^{1,n} \cap (I_{j}^{2,n})_{-\tilde{\theta}} \neq \emptyset } }{ \sqrt{ \sum_{i:\overline{I_{i}^{1,n}}\leq T} B^{1}((I_{i}^{n})_{\theta})^{2} } \sqrt{ \sum_{j} B^{2}(I_{j}^{2,n})^{2} } } \\
& \quad + \mathbf{1}_{\Theta_{<0}}(\tilde{\theta}) \frac{ \sum_{i,j:\overline{I_{j}^{2,n}}\leq T} B^{1}((I_{i}^{1,n})_{\theta}) B^{2}(I_{j}^{2,n}) \mathbf{1}_{ (I_{i}^{1,n})_{\tilde{\theta}} \cap I_{j}^{2,n} \neq \emptyset } }{ \sqrt{ \sum_{i} B^{1}((I_{i}^{n})_{\theta})^{2} } \sqrt{ \sum_{j:\overline{I_{j}^{2,n}}\leq T} B^{2}(I_{j}^{2,n})^{2} } }
\end{align}
if the denominators appearing in (\ref{calc8}) are not equal zero. Note that they are not equal zero $\mathbb{P}$-almost surely. We start with the following lemma.

\begin{lemma}\label{lemma1} 
Let $J$ denote $ [0,T] $ or $ [0,T+\delta] $. Then we have, for all $ a\in[0,\delta) $,
\begin{align}\label{limit1} 
\frac{ \sum_{i: \overline{I_{i}^{l,n}}\in J} B^{l}((I_{i}^{l,n})_{a})^{2} }{ \sum_{i: \overline{I_{i}^{l,n}} \in J}|I_{i}^{l,n}|^{2H_{l}} } \to^{p} 1
\end{align}
as $n\to\infty$ for $l=1,2$.
\end{lemma}
\begin{proof}
First we define an auxiliary set $ A_{n}(K) $ by 
\begin{align*}
A_{n}(K) = \bigcap_{l=1,2} \left\{  \frac{r_{n}^{2H_{l}-1+\mu}v_{n}^{-\gamma(\mu)}}{\sum_{i:\overline{I_{i}^{l,n}} \leq T }| I_{i}^{l,n} |^{2H_{l}}} \leq K \right\}. 
\end{align*}
Since $ \sum_{i:\overline{I_{i}^{l,n}} \leq T }| I_{i}^{l,n} |^{2H_{l}} \leq T r_{n}^{2H_{l}-1}$ holds, we have 
\[
r_{n} \leq (TK)^{1/\mu} v_{n}^{\gamma(\mu)/\mu}, 
\]
on the set $A_{n}(K)$. In particular, it holds that $ ( r_{n}/v_{n} )1_{A_{n}(K)} \to 0 $ as $n\to\infty$ (recall that $\gamma(\mu) >\mu$). Thanks to (B3), it holds that $ \sup_{n}\mathbb{P}(A_{n}(K)^{c}) < \epsilon $ for each $\epsilon>0$ if $ K= K(\epsilon) $ is sufficiently large. 
Let $R^{l,n}_{1} $ denote $ (\sum_{i : \overline{I_{i}^{l,n}} \in J}B^{l}((I_{i}^{l,n})_{a})^{2})/(\sum_{i:\overline{I_{i}^{l,n}} \in J}|I_{i}^{l,n}|^{2H_{l}} )$ for simplicity. For any positive number $\epsilon_{1} >0 $, we have
\begin{align*}
\mathbb{P}\left\{ \left| R^{l,n}_{1} -1 \right| \geq \epsilon_{1} \right\} &\leq \mathbb{P}\left\{ \left| R^{l,n}_{1} -1 \right| \geq \epsilon_{1} ,\ A_{n}(K) \right\}  + \mathbb{P}\{ A_{n}(K)^{c} \}.
\end{align*}
Hence we obtain 
\begin{align} \label{ineq16}
\nonumber \mathbb{P}\left\{ \left| R^{l,n}_{1} -1 \right| \geq \epsilon_{1} \right\} &\leq \mathbb{P}\left\{ \left| R^{l,n}_{1} -1 \right| \geq\epsilon_{1},\ A_{n}(K(\epsilon_{2})) \right\} + \epsilon_{2} \\
&\leq \epsilon_{1}^{-2} \mathbb{E}\left\{ \left| R^{l,n}_{1} -1 \right|^{2} \mathbf{1}_{A_{n}(K(\epsilon_{2})) } \right\} + \epsilon_{2}. 
\end{align}
By using [B1] and conditioning, we can calculate the expectation in (\ref{ineq16}) as
\begin{align*}
& \mathbb{E}\left\{ \left| R^{l,n}_{1} -1 \right|^{2} \mathbf{1}_{A_{n}(K(\epsilon_{2})) } \right\}\\
&= \mathbb{E}\left\{ \mathbb{E}\left\{ \left( \frac{ \sum_{i : \overline{I_{i}^{l,n}} \in J}( B^{l}( (I_{i}^{l,n})_{a} )^{2} - | I_{i}^{l,n}  |^{2H_{l}}   )}{ \sum_{i : \overline{I_{i}^{l,n}} \in J}| I_{i}^{l,n} |^{2H_{l}} }  \right)^{2} \mathbf{1}_{A_{n}(K(\epsilon_{2})) } \middle| \sigma(\mathcal{T}) \right\} \right\} \\
&= \mathbb{E}
 \left\{
 \frac{ \mathbf{1}_{A_{n}(K(\epsilon_{2}))} \mathbb{E}
 \left\{
 \left( \sum_{i : \overline{I_{i}^{l,n}} \in J}\left(B^{l}((I_{i}^{l,n})_{a} )^{2} - | I_{i}^{l,n} |^{2H_{l}} \right)  \right)^{2} 
 \middle| \sigma(\mathcal{T})
 \right\}  }{ (\sum_{i : \overline{I_{i}^{l,n}} \in J}|I_{i}^{l,n} |^{2H_{l}})^{2} }
 \right\}.
\end{align*} 

Since $ B^{{l}}((I_{i}^{l,n})_{a} )^{2} - | (I_{i}^{l,n}) |^{2H_{l}} = B^{{l}}((I_{i}^{l,n})_{a} )^{2} - | (I_{i}^{l,n})_{a} |^{2H_{l}} $ coincides with the multiple Wiener integral $ \mathbb{I}_{2}( h^{l}( (I_{i}^{l,n})_{a} )^{\otimes 2}  ) $ (recall that the definition of $h^{l}(I)$ is given in (\ref{def1})) by Theorem \ref{theorem2}, we have 
\begin{align*}
 \mathbb{E}
 \left\{
 \left( \sum_{i : \overline{I_{i}^{l,n}} \in J}\left(B^{l}((I_{i}^{l,n})_{a})^{2} - | I_{i}^{l,n} |^{2H_{l}} \right)  \right)^{2} 
 \middle| \sigma(\mathcal{T})
 \right\}
 &= \mathbb{E}\left\{ \mathbb{I}_{2}^{2} \left( \sum_{i : \overline{I_{i}^{l,n}} \in J} h^{l}( (I_{i}^{l,n})_{a} )^{\otimes 2} \right) \middle| \sigma(\mathcal{T}) \right\} \\
 &= \sum_{i,j} \langle h^{l}( (I_{i}^{l,n})_{a} ) , h^{l}( (I_{j}^{l,n})_{a} ) \rangle_{L^{2}([0,T+2\delta];\mathbb{R}^{2})}^{2} \\ 
 &= \sum_{i,j} \langle \mathbf{1}_{ (I_{i}^{l,n})_{a} } , \mathbf{1}_{ (I_{j}^{l,n})_{a} } \rangle_{\mathcal{H}_{H_{l}} }
\end{align*}
Since $ | \langle \mathbf{1}_{ (I_{i}^{l,n})_{a} } , \mathbf{1}_{ (I_{j}^{l,n})_{a} } \rangle_{\mathcal{H}_{H_{l}} } | \leq r_{n}^{2H_{l}} $ because of Cauchy-Schwarz inequality, we have 
\begin{align*}
 \mathbb{E}
 \left\{
 \left( \sum_{i : \overline{I_{i}^{l,n}} \in J}\left(B^{l}((I_{i}^{l,n})_{a})^{2} - | I_{i}^{l,n} |^{2H_{l}} \right)  \right)^{2} 
 \middle| \sigma(\mathcal{T})
 \right\}
&\leq r_{n}^{2H_{l}} \sum_{i,j} \langle \mathbf{1}_{ (I_{i}^{l,n})_{a} } , \mathbf{1}_{ (I_{j}^{l,n})_{a} } \rangle_{\mathcal{H}_{H_{l}} } \\
 &\lesssim r_{n}^{2H_{l}}.
\end{align*}
Note that $ \langle \mathbf{1}_{ (I_{i}^{l,n})_{a} } , \mathbf{1}_{ (I_{j}^{l,n})_{a} } \rangle_{\mathcal{H}_{H_{l}} } $ is non-negative. Plugging this into (\ref{ineq16}), we obtain 
\begin{align*}
\mathbb{P}\left\{ \left| R^{l,n}_{1} -1 \right| \geq \epsilon_{1} \right\} &\lesssim \epsilon_{1}^{-2} \mathbb{E}\left\{ \left( \frac{ r_{n}^{H_{l} } }{ \sum_{i : \overline{I_{i}^{l,n}} \in J}| I_{i}^{l,n} |^{2H_{l}} } \right)^{2} \mathbf{1}_{A_{n}(K(\epsilon_{2}))} \right\} + \epsilon_{2} \\
&=  \epsilon_{1}^{-2} \mathbb{E}\left\{ \left( \frac{ r_{n}^{2H_{l}-1+\mu } v_{n}^{-\gamma(\mu)} }{ \sum_{i : \overline{I_{i}^{l,n}} \in J}| I_{i}^{l,n} |^{2H_{l}} } r_{n}^{1-H_{l}-\mu} v_{n}^{\gamma(\mu)} \right)^{2} \mathbf{1}_{A_{n}(K(\epsilon_{2}))} \right\} + \epsilon_{2} \\
&\lesssim \epsilon_{1}^{-2} (TK(\epsilon_{2}))^{2/\mu} v_{n}^{2(1 -H_{l})} + \epsilon_{2},
\end{align*}
if we choose sufficiently small $\mu>0$ such that $1-H-\mu>0$ holds. Since $\epsilon_{1}$ and $\epsilon_{2}$ are arbitrary, we obtain (\ref{limit1}) by letting $n\to\infty$.
\end{proof}

Let us set
\begin{align*}
\nonumber R_{2}^{n}(\tilde{\theta}) &= \mathbf{1}_{\Theta_{\geq0}}(\tilde{\theta})\sum_{i,j:\overline{I_{i}^{1,n}}\leq T} B^{1}((I_{i}^{1,n})_{\theta})B^{2}(I_{j}^{2,n})\mathbf{1}_{I_{i}^{1,n}\cap (I_{j}^{2,n})_{-\tilde{\theta}}\neq \emptyset} \\
& \quad + \mathbf{1}_{\Theta_{<0}}(\tilde{\theta})\sum_{i,j:\overline{I_{j}^{2,n}}\leq T} B^{1}((I_{i}^{1,n})_{\theta})B^{2}(I_{j}^{2,n})\mathbf{1}_{(I_{i}^{1,n})_{\tilde{\theta}}\cap (I_{j}^{2,n})\neq \emptyset},
\end{align*}
and
\begin{align*}
\nonumber D^{n}(\tilde{\theta}) &= \mathbf{1}_{\Theta_{\geq0}}(\tilde{\theta})\sqrt{\sum_{i:\overline{I_{i}^{1,n}}\leq T}| I_{i}^{1,n} |^{2H_{1}} } \sqrt{ \sum_{j}|I_{j}^{2,n}|^{2H_{2}} } 
\\& \quad + \mathbf{1}_{\Theta_{<0}}(\tilde{\theta})\sqrt{\sum_{i}| I_{i}^{1,n} |^{2H_{1}} } \sqrt{ \sum_{j:\overline{I_{j}^{2,n}}\leq T}|I_{j}^{2,n}|^{2H_{2}} }.
\end{align*}

Lemma \ref{lemma1} reduces calculation of $\mathcal{U}_{n}(\tilde{\theta})$ to that of $R_{2}^{n}(\tilde{\theta})/D^{n}(\tilde{\theta})$. Next we show that $R_{2}^{n}(\tilde{\theta})/D^{n}(\tilde{\theta})$ goes to zero as $n\to\infty$ if $\tilde{\theta}$ is distant from $\theta$.

\begin{proposition}\label{lemma2}
We have, for any $\epsilon>0$,
\begin{align}\label{limit3} 
\sup_{\mathcal{G}^{n} \cap \{\tilde{\theta}\in\Theta \mid |\tilde{\theta}-\theta|\geq\epsilon v_{n}\} }\left| \frac{R^{n}_{2}(\tilde{\theta})}{D^{n}(\tilde{\theta})} \right| \to^{p} 0
\end{align}
as $n\to\infty$.
\end{proposition}
\begin{proof}
The set $ \mathcal{G}^{n}\cap\{\tilde{\theta}\in\Theta\mid |\tilde{\theta}-\theta|\geq\epsilon v_{n}\} $ can be decomposed as
\begin{align*}
\nonumber  \mathcal{G}^{n}\cap\{\tilde{\theta}\in\Theta \mid|\tilde{\theta}-\theta|\geq\epsilon v_{n}\} &= ( \mathcal{G}^{n}_{\geq0}\cap\{\tilde{\theta}\in\Theta \mid \tilde{\theta} \geq \theta + \epsilon v_{n} \} ) \cup ( \mathcal{G}^{n}_{\geq0}\cap\{\tilde{\theta}\in\Theta \mid \tilde{\theta} \leq \theta - \epsilon v_{n} \} ) \\
\nonumber & \quad \cup (\mathcal{G}^{n}_{<0}\cap \{ \tilde{\theta}\in\Theta \mid \tilde{\theta} \leq \theta - \epsilon v_{n} \}) \\
&=: \mathcal{G}^{n}_{1} \cup \mathcal{G}^{n}_{2} \cup \mathcal{G}^{n}_{3}.
\end{align*}
Clearly (\ref{limit3}) is equivalent to 
\begin{align}\label{limit4} 
\sup_{\mathcal{G}^{n}_{i}}\left| \frac{R^{n}_{2}(\tilde{\theta})}{D^{n}(\tilde{\theta})} \right| \to^{p} 0
\end{align}
for $i=1,2,3$. We only prove $(\ref{limit4})$ when $i=1$. The other cases are similar.

Let us assume $ H_{1} \leq H_{2} $ for the moment. Since $ \tilde{\theta} \in \mathcal{G}^{n}_{1} $, we have 
\begin{align*}
\nonumber R_{2}^{n}(\tilde{\theta}) &= \sum_{i,j:\overline{I_{i}^{1,n}}\leq T} B^{1}( (I_{i}^{1,n})_{\theta} )B^{2}((I_{j}^{2,n}))\mathbf{1}_{ (I_{i}^{1,n})_{\theta}\cap (I_{j}^{2,n})_{\theta-\tilde{\theta}}\neq \emptyset} \\
\nonumber &= \sum_{i:\overline{I_{i}^{1,n}}\leq T} B^{1}( (I_{i}^{1,n})_{\theta} )B^{2}(\mathcal{I}^{2,n}_{\theta-\tilde{\theta}}((I_{i}^{1,n})_{\theta})_{\tilde{\theta}-\theta}) \\
\nonumber &= \sum_{i:\overline{I_{i}^{1,n}}\leq T} \langle h^{1}( (I_{i}^{1,n})_{\theta} ) , h^{2}( \mathcal{I}^{2,n}_{\theta-\tilde{\theta}}((I_{i}^{1,n})_{\theta})_{\tilde{\theta}-\theta} ) \rangle_{L^{2}([0,T+2\delta];\mathbb{R}^{2})} \\
\nonumber &\quad +\sum_{i:\overline{I_{i}^{1,n}}\leq T} \Bigl( B^{1}( (I_{i}^{1,n})_{\theta} )B^{2}(\mathcal{I}^{2,n}_{\theta-\tilde{\theta}}((I_{i}^{1,n})_{\theta})_{\tilde{\theta}-\theta}) \\
\nonumber &\quad- \langle h^{1}( (I_{i}^{1,n})_{\theta} ) , h^{2}( \mathcal{I}^{2,n}_{\theta-\tilde{\theta}}((I_{i}^{1,n})_{\theta})_{\tilde{\theta}-\theta} ) \rangle_{L^{2}([0,T+2\delta];\mathbb{R}^{2})} \Bigr) \\
&=: R_{3}^{n}(\tilde{\theta}) + R_{4}^{n}(\tilde{\theta}).
\end{align*}

Hence it suffices to show
\begin{align}\label{limit20}
\sup_{\tilde{\theta}\in\mathcal{G}^{n}_{1}} \left| \frac{R_{3}^{n}(\tilde{\theta})}{D^{n}(\tilde{\theta})} \right| \to^{p} 0
\end{align}
and
\begin{align}\label{limit21}
\sup_{\tilde{\theta} \in \mathcal{G}_{1}^{n} }\left| \frac{R_{4}^{n}(\tilde{\theta})}{D^{n}(\tilde{\theta})} \right| \to^{p} 0
\end{align}
as $n\to\infty$. 

\textbf{The limit (\ref{limit20}).} 
It suffices to show 
\begin{align}\label{limit2} 
\sup_{\tilde{\theta}\in\mathcal{G}^{n}_{1}} \left| \frac{R_{3}^{n}(\tilde{\theta})}{D^{n}(\tilde{\theta})} \right| \mathbf{1}_{A_{n}(K)} \to 0
\end{align}
as $n\to\infty$ for each $K>0$. By (\ref{calc4}) in Proposition \ref{fBM2}, it holds that 
\begin{align*}
\nonumber |R_{3}^{n}(\tilde{\theta})| &= |\rho| c_{H_{1}}c_{H_{2}} \sum_{i:\overline{I_{i}^{1,n}}\leq T} \int_{ (I_{i}^{1,n})_{\theta} }du\int_{\mathcal{I}^{2,n}_{\theta-\tilde{\theta}}((I_{i}^{1,n})_{\theta})_{\tilde{\theta}-\theta}}dv\ \beta(u,v) \\&\quad \times |u-v|^{H_{1}+H_{2}-2} u^{H_{1}-H_{2}} v^{H_{2}-H_{1}}.
\end{align*}
Note that for each $K>0$ and $\epsilon>0$ we can choose $n_{0}=n_{0}(K,\epsilon)$ such that $n\geq n_{0}$ implies 
\begin{align*}
(TK)^{1/\mu}v_{n}^{(\gamma(\mu)/\mu)-1} \leq \frac{\epsilon}{3}.
\end{align*}
Hence if $n\geq n_{0}$ then $r_{n}\mathbf{1}_{A_{n}(K)} \leq (\epsilon v_{n}/3)$ holds. We can always assume $n\geq n_{0}$ in this proof. 
Since $ \tilde{\theta} \in \mathcal{G}^{n}_{1} $ and $r_{n}\leq(\epsilon v_{n})/3$ on $A_{n}(K)$, we have 
$|u-v| \geq (\epsilon v_{n})/3$. Therefore 
\begin{align} \label{ineq3} 
| R_{3}^{n}(\tilde{\theta}) |\lesssim (\epsilon v_{n})^{H_{1}+H_{2}-2} \sum_{i:\overline{I_{i}^{1,n}}\leq T} \int_{ (I_{i}^{1,n})_{\theta} }du\ u^{H_{1}-H_{2}} \int_{\mathcal{I}^{2,n}_{\theta-\tilde{\theta}}((I_{i}^{1,n})_{\theta})_{\tilde{\theta}-\theta}}dv\ v^{H_{2}-H_{1}}
\end{align}
on $A_{n}(K)$. Since $H_{1}\leq H_{2}$, we have 
\begin{align} \label{ineq8}
\int_{\mathcal{I}^{2,n}_{\theta-\tilde{\theta}}((I_{i}^{1,n})_{\theta})_{\tilde{\theta}-\theta}}dv\ v^{H_{2}-H_{1}} \lesssim r_{n}.
\end{align}
Since the integral $ \int_{ (I_{i}^{1,n})_{\theta} }du\ u^{H_{1}-H_{2}} $ is summable (note that $ H_{1}-H_{2} \in (-1/2,1/2) $), we have 
\begin{align}\label{ineq1} 
| R_{3}^{n}(\tilde{\theta}) |\lesssim (\epsilon v_{n})^{H_{1}+H_{2}-2} r_{n}.
\end{align}
By (\ref{ineq1}) and (B3), we have
\begin{align}\label{ineq2} 
\left| \frac{R_{3}^{n}(\tilde{\theta})}{D^{n}(\tilde{\theta})} \right| \mathbf{1}_{A_{n}(K)} \lesssim v_{n}^{\gamma(\mu)-\mu} \mathbf{1}_{A_{n}(K)}.
\end{align} 
Since the right-hand-side of (\ref{ineq2}) does not depend on $\tilde{\theta}$, we obtain (\ref{limit2}).

\textbf{The limit (\ref{limit21}).} 
We define $ \tilde{h}_{i}^{n}(\tilde{\theta}) \in L^{2}([0,T+2\delta];\mathbb{R}^{2}) \tilde{\otimes} L^{2}([0,T+2\delta];\mathbb{R}^{2}) $ by
\begin{align*}
\tilde{h}_{i}^{n}(\tilde{\theta})(\omega) = h^{1}( (I_{i}^{1,n})_{\theta}(\omega) ) \tilde{\otimes} h^{2}( \mathcal{I}^{2,n}_{\theta-\tilde{\theta}}((I_{i}^{1,n})_{\theta})_{\tilde{\theta}-\theta}(\omega) ).
\end{align*}
Let $ \epsilon_{1} $ and $\epsilon_{2}$ be positive numbers. By the same conditioning technique as in the proof of Lemma \ref{lemma1}, we have 
\begin{align}\label{ineq6} 
\mathbb{P}\left\{ \sup_{\tilde{\theta} \in \mathcal{G}^{n}_{1} }\left| \frac{R_{4}^{n}(\tilde{\theta})}{D^{n}(\tilde{\theta})} \right| \geq \epsilon_{1} \right\} \leq \epsilon_{2}  + \epsilon_{1}^{-2p}\sum_{ \tilde{\theta} \in \mathcal{G}^{n}_{1} } \mathbb{E}\left\{ \frac{\mathbf{1}_{A_{n}(K(\epsilon_{2}))} \mathbb{E}\left\{ \left| \mathbb{I}_{2}(\sum_{i:\overline{I_{i}^{1,n}}\leq T} \tilde{h}^{n}_{i}(\tilde{\theta}) ) \right|^{2p} \middle| \sigma(\mathcal{T}) \right\} }{|D^{n}(\tilde{\theta})|^{2p}} \right\}
\end{align}
for any $p\geq 1$. Using Theorem \ref{theorem3}, we obtain 
\begin{align}\label{ineq7} 
 \mathbb{E}\left\{ \left| \mathbb{I}_{2}\left(\sum_{i:\overline{I_{i}^{1,n}}\leq T} \tilde{h}^{n}_{i}(\tilde{\theta})\right) \right|^{2p}\middle| \sigma(\mathcal{T}) \right\} \leq C_{p} \mathbb{E} \left\{ \left| \mathbb{I}_{2}\left(\sum_{i:\overline{I_{i}^{1,n}}\leq T} \tilde{h}^{n}_{i}(\tilde{\theta})\right) \right|^{2} \middle| \sigma(\mathcal{T}) \right\}^{p}
\end{align}
where $C_{p}$ is a positive constant depending only on $p\geq1$. Let $R_{5}^{n}(\tilde{\theta})$ denote the expectation in the right-hand-side of (\ref{ineq7}). A simple calculation using  Proposition \ref{proposition3} yields
\begin{align*}
\nonumber R_{5}^{n}(\tilde{\theta}) &= \sum_{i,k:\overline{I_{i}^{1,n}}\leq T,\ \overline{I_{k}^{2,n}}\leq T} \langle h^{1}( (I_{i}^{1,n})_{\theta} ), h^{1}( (I_{k}^{1,n})_{\theta} ) \rangle_{L^{2}([0,T+2\delta] ; \mathbb{R}^{2})} \\
\nonumber &\quad\times \langle h^{2}( \mathcal{I}^{2,n}_{\theta-\tilde{\theta}}((I_{i}^{1,n})_{\theta})_{\tilde{\theta}-\theta}, h^{2}( \mathcal{I}^{2,n}_{\theta-\tilde{\theta}}((I_{k}^{1,n})_{\theta})_{\tilde{\theta}-\theta} \rangle_{L^{2}([0,T+2\delta] ; \mathbb{R}^{2})} \\
\nonumber &\quad + \sum_{i,k:\overline{I_{i}^{1,n}}\leq T,\ \overline{I_{k}^{2,n}}\leq T} \langle h^{1}( (I_{i}^{1,n})_{\theta} ),  h^{2}( \mathcal{I}^{2,n}_{\theta-\tilde{\theta}}((I_{k}^{1,n})_{\theta})_{\tilde{\theta}-\theta} \rangle_{L^{2}([0,T+2\delta] ; \mathbb{R}^{2})} \\
\nonumber &\quad\times \langle h^{1}( (I_{k}^{1,n})_{\theta} ),  h^{2}( \mathcal{I}^{2,n}_{\theta-\tilde{\theta}}((I_{i}^{1,n})_{\theta})_{\tilde{\theta}-\theta} \rangle_{L^{2}([0,T+2\delta] ; \mathbb{R}^{2})} \\
&\quad =: R_{5,1}^{n}(\tilde{\theta}) + R_{5,2}^{n}(\tilde{\theta}). 
\end{align*}

(1) The first term can be estimated as follows. Since 
\begin{align*}
\langle h^{2}( \mathcal{I}^{2,n}_{\theta-\tilde{\theta}}((I_{i}^{1,n})_{\theta})_{\tilde{\theta}-\theta}, h^{2}( \mathcal{I}^{2,n}_{\theta-\tilde{\theta}}((I_{k}^{1,n})_{\theta})_{\tilde{\theta}-\theta} \rangle_{L^{2}([0,T+2\delta] ; \mathbb{R}^{2})} \lesssim r_{n}^{2H_{2}},
\end{align*}
we have 
\begin{align}\label{ineq4} 
| R_{5,1}^{n}(\tilde{\theta}) | \lesssim r_{n}^{2H_{2}} 
\end{align}
on $A_{n}(K(\epsilon_{2}))$. Note that $ I \mapsto h^{1}(I) $ is linear so that 
\[
\sum_{i,k: \overline{I_{i}^{1,n}} \leq T,\ \overline{I_{k}^{1,n}}\leq T} \langle h^{1}( (I_{i}^{1,n})_{\theta} ), h^{1}( (I_{k}^{1,n})_{\theta} ) \rangle_{L^{2}([0,T+2\delta] ; \mathbb{R}^{2})}  < \infty.
\]

(2) The second term can be estimated as follows. We first recall that 
\begin{align*}
&\quad\nonumber | \langle h^{1}( (I_{i}^{1,n})_{\theta} ),  h^{2}( \mathcal{I}^{2,n}_{\theta-\tilde{\theta}}((I_{k}^{1,n})_{\theta})_{\tilde{\theta}-\theta} \rangle_{L^{2}([0,T+2\delta] ; \mathbb{R}^{2})} | 
\\ \nonumber &= |\rho| c_{H_{1}}c_{H_{2}}  \int_{ (I_{i}^{1,n})_{\theta} }du\int_{\mathcal{I}^{2,n}_{\theta-\tilde{\theta}}((I_{k}^{1,n})_{\theta})_{\tilde{\theta}-\theta}}dv\ \beta(u,v) \\&\quad \times |u-v|^{H_{1}+H_{2}-2} u^{H_{1}-H_{2}} v^{H_{2}-H_{1}}.
\end{align*}
By the same reasoning as (\ref{ineq3}), we have
\begin{align*}
\langle h^{1}( (I_{i}^{1,n})_{\theta} ),  h^{2}( \mathcal{I}^{2,n}_{\theta-\tilde{\theta}}((I_{k}^{1,n})_{\theta})_{\tilde{\theta}-\theta} \rangle_{L^{2}([0,T+2\delta] ; \mathbb{R}^{2})} \lesssim \left( \frac{\epsilon v_{n}}{3} \right)^{H_{1}+H_{2}-2} r_{n} | I_{i}^{1,n} |
\end{align*} 
if $ i\leq k $, and 
\begin{align*}
\langle h^{1}( (I_{k}^{1,n})_{\theta} ),  h^{2}( \mathcal{I}^{2,n}_{\theta-\tilde{\theta}}((I_{i}^{1,n})_{\theta})_{\tilde{\theta}-\theta} \rangle_{L^{2}([0,T+2\delta] ; \mathbb{R}^{2})} \lesssim \left( \frac{\epsilon v_{n}}{3} \right)^{H_{1}+H_{2}-2} r_{n} | I_{k}^{1,n} |
\end{align*} 
if  $ i>k $, on $ A_{n}(K(\epsilon_{2})) $. Therefore, it holds that
\begin{align}\label{ineq5} 
|\nonumber R_{5,2}^{n}(\tilde{\theta})| &\lesssim r_{n}v_{n}^{H_{1}+H_{2}-2} \Biggl( \sum_{i\leq k } |I_{i}^{1,n}| |\langle h^{1}( (I_{k}^{1,n})_{\theta} ),  h^{2}(  \mathcal{I}^{2,n}_{\theta-\tilde{\theta}}((I_{i}^{1,n})_{\theta})_{\tilde{\theta}-\theta} \rangle_{L^{2}([0,T+2\delta] ; \mathbb{R}^{2})} | \\
\nonumber &\quad+ \sum_{k>i} |I_{k}^{1,n}| |\langle h^{1}( (I_{i}^{1,n})_{\theta} ),  h^{2}(  \mathcal{I}^{2,n}_{\theta-\tilde{\theta}}((I_{k}^{1,n})_{\theta})_{\tilde{\theta}-\theta} \rangle_{L^{2}([0,T+2\delta] ; \mathbb{R}^{2})} | \Biggr) \\
&\lesssim r_{n}^{1+H_{2}} v_{n}^{H_{1}+H_{2}-2}
\end{align}
on $A_{n}(K(\epsilon_{2}))$. Plugging (\ref{ineq4}) and (\ref{ineq5}) into (\ref{ineq6}), we obtain 
\begin{align*}
\mathbb{P}\left\{ \sup_{\tilde{\theta} \in \mathcal{G}_{1}^{n} }\left| \frac{R_{4}^{n}(\tilde{\theta})}{D^{n}(\tilde{\theta})} \right| \geq \epsilon_{1} \right\} &\lesssim \epsilon_{2} + \epsilon_{1}^{-2p} \sum_{ \tilde{\theta} \in \mathcal{G}_{1}^{n} } \mathbb{E}\left\{ \frac{ \mathbf{1}_{A_{n}(K(\epsilon_{2}))} |r_{n}^{2H_{2}} + r_{n}^{1+H_{2}}v_{n}^{H_{1}+H_{2}-2} |^{p} }{ |D^{n}(\tilde{\theta})|^{2p}  } \right\} \\ 
&\lesssim \epsilon_{2} + \epsilon_{1}^{-2p} (\# \mathcal{G}^{n}) v_{n}^{p(1-H_{1})}
\end{align*}
by (B3). Thanks to (C2), we have $\lim_{n\to\infty}(\# \mathcal{G}^{n}) v_{n}^{p(1-H_{1})}= 0$ if we choose sufficiently large $p\geq1$. This gives (\ref{limit21}).

Thanks to (\ref{limit20}) and (\ref{limit21}), we obtain $(\ref{limit4})$ when $ H_{1} \leq H_{2} $. Now let us consider the case where $ H_{1} > H_{2} $. In the case where $ H_{1} > H_{2} $, we use an alternative expression of $ R_{2}^{n}(\tilde{\theta}) $:
\begin{align}\label{calc5}
\nonumber  R_{2}^{n}(\tilde{\theta}) &= \sum_{i,j: \overline{I_{i}^{1,n}} \leq T} B^{1}((I_{i}^{1,n})_{\theta})B^{2}( I_{j}^{2,n} ) \mathbf{1}_{ (I_{i}^{1,n})_{\tilde{\theta}} \cap I_{j}^{2,n} \neq \emptyset } \\
&= \sum_{j} B^{2}(I_{j}^{2,n}) B^{1} ( \mathcal{I}_{\tilde{\theta}}^{1,n,\leq T} ( I_{j}^{2,n} )_{\theta-\tilde{\theta}} ),
\end{align}
where the symbol $\mathcal{I}_{\tilde{\theta}}^{1,n,\leq T}$ denotes the family of shifted intervals
\begin{align*}
\mathcal{I}_{\tilde{\theta}}^{1,n,\leq T} = \{ (I_{i}^{1,n})_{\tilde{\theta}} \mid \overline{I_{i}^{1,n}} \leq T \}.
\end{align*}
Using the expression (\ref{calc5}), we can prove (\ref{limit4}) when $ H_{1} > H_{2} $ in the same manner as we did in the case where $ H_{1} \leq H_{2} $.
\end{proof}



Let us show that $R_{2}^{n}(\tilde{\theta}_{n}) / D^{n}(\tilde{\theta}_{n})$ is away from zero if $|\tilde{\theta}_{n} - \theta|$ decreases rapidly as $n\to\infty$.

\begin{proposition} \label{lemma5} 
Suppose that $\tilde{\theta}_{n}$ satisfies $ | \theta- \tilde{\theta}_{n} | \leq \rho_{n} $ for all $n\in\mathbb{N}$. Then there exists $ c_{\ast}^{\prime} >0 $ such that
\begin{align} \label{limit6}
\mathbb{P} \left\{ \left| \frac{R_{2}^{n}(\tilde{\theta}_{n})}{D^{n}(\tilde{\theta}_{n})} \right| \geq c_{\ast}^{\prime} \right\} \to 1
\end{align}
as $n\to\infty$.
\end{proposition}
\begin{proof}
We formally extend $B^{1}$ and $B^{2}$ by setting $ B^{1}_{t} = B^{2}_{t} =0 $ for $t<0$. Let us first suppose that $\tilde{\theta}_{n} \in \Theta_{\geq0} $. This is the case for sufficiently large $n$ if $\theta\in(0,\delta)$. Then we decompose $R_{2}^{n}(\tilde{\theta}_{n})$ into three terms:
\begin{align} \label{calc6} 
\nonumber R_{2}^{n}(\tilde{\theta}_{n}) &= \sum_{i:\overline{I_{i}^{1,n}}\leq T} B^{1}( (I_{i}^{1,n})_{\theta} )\left( B^{2}( \mathcal{I}^{2,n}_{\theta-\tilde{\theta}_{n}} ( (I_{i}^{1,n})_{\theta} )_{\tilde{\theta}_{n}-\theta} )- B^{2}( \mathcal{I}^{2,n}_{\theta-\tilde{\theta}_{n}} ( (I_{i}^{1,n})_{\theta} )\cap[0,T+2\delta]) \right) \\
\nonumber &\quad+ \sum_{i:\overline{I_{i}^{1,n}}\leq T} \Bigl(B^{1}( (I_{i}^{1,n})_{\theta} ) B^{2}( \mathcal{I}^{2,n}_{\theta-\tilde{\theta}_{n}} ( (I_{i}^{1,n})_{\theta} ) \cap[0,T+2\delta] ) 
\\ \nonumber &\quad- \langle h^{1}( (I_{i}^{1,n})_{\theta} ) , h^{2}( \mathcal{I}^{2,n}_{\theta-\tilde{\theta}_{n}}((I_{i}^{1,n})_{\theta})\cap [0,T+2\delta] ) \rangle_{L^{2}([0,T+2\delta];\mathbb{R}^{2})} \Bigr) \\
\nonumber &\quad + \sum_{i:\overline{I_{i}^{1,n}}\leq T}\langle h^{1}( (I_{i}^{1,n})_{\theta} ) , h^{2}( \mathcal{I}^{2,n}_{\theta-\tilde{\theta}_{n}}((I_{i}^{1,n})_{\theta})\cap[0,T+2\delta] ) \rangle_{L^{2}([0,T+2\delta];\mathbb{R}^{2})} \\
&=: R_{6}^{n}(\tilde{\theta}_{n}) + R_{7}^{n}(\tilde{\theta}_{n}) + R_{8}^{n}(\tilde{\theta}_{n}).
\end{align}
To simplify the notation, we set $ \mathcal{I}^{2,n}_{\theta-\tilde{\theta}_{n}}((I_{i}^{1,n})_{\theta})^{+} = \mathcal{I}^{2,n}_{\theta-\tilde{\theta}_{n}}((I_{i}^{1,n})_{\theta})\cap[0,T+2\delta] $. Let us show
\begin{align} \label{limit22}
\left| \frac{R_{6}^{n}(\tilde{\theta}_{n})}{D^{n}(\tilde{\theta}_{n})} \right| \to^{p} 0,
\end{align}
\begin{align} \label{limit23}
\left| \frac{R_{7}^{n}(\tilde{\theta}_{n})}{D^{n}(\tilde{\theta}_{n})} \right| \to^{p} 0
\end{align}
and that there exists $ c_{\ast}>0 $ such that 
\begin{align}\label{limit24}
\mathbb{P} \left\{ \left| \frac{R_{8}^{n}(\tilde{\theta}_{n})}{D^{n}(\tilde{\theta}_{n})} \right| \geq c_{\ast}^{\prime} \right\} \to 1
\end{align}
as $n\to\infty$.

\textbf{The limit (\ref{limit22})}.
Thanks to (B3), it suffices to show 
\begin{align}\label{limit5}
\left| \frac{R_{6}^{n}(\tilde{\theta}_{n})}{D^{n}(\tilde{\theta}_{n})} \right| \mathbf{1}_{A_{n}(K)} \to^{p} 0
\end{align}
as $n\to\infty$. By the conditioning argument as in the proof of Lemma \ref{lemma1}, we have
\begin{align}\label{eq1}
\nonumber \mathbb{E}\left\{  \frac{|R_{6}^{n}(\tilde{\theta_{n}})|\mathbf{1}_{A_{n}(K)}}{|D^{n}(\tilde{\theta}_{n})|} \right\} &\leq \mathbb{E}\Biggl\{ \frac{\mathbf{1}_{A_{n}(K)} }{ | D^{n} | } \sum_{i} \mathbb{E}\{  | B^{1}((I_{i}^{1,n})_{\theta})| \\
&\quad \times | B^{2}(\mathcal{I}^{2,n}_{\theta-\tilde{\theta}_{n}}((I_{i}^{1,n})_{\theta})_{\tilde{\theta}_{n}-\theta}) - B^{2}(\mathcal{I}^{2,n}_{\theta-\tilde{\theta}_{n}}((I_{i}^{1,n})_{\theta})^{+}) | \mid \sigma(\mathcal{T}) \} \Biggr\}.
\end{align}
Let $p>1$ and $q>1$ be conjugate indices: $1/p + 1/q = 1$. Using (conditional) H\"{o}lder's inequality and Theorem \ref{theorem3}, we have
\begin{align}\label{ineq17}
\nonumber &\sum_{i} \mathbb{E}\{  | B^{1}((I_{i}^{1,n})_{\theta})| | B^{2}(\mathcal{I}^{2,n}_{\theta-\tilde{\theta}_{n}}((I_{i}^{1,n})_{\theta})_{\tilde{\theta}_{n}-\theta}) - B^{2}(\mathcal{I}^{2,n}_{\theta-\tilde{\theta}_{n}}((I_{i}^{1,n})_{\theta})^{+}) | \mid \sigma(\mathcal{T}) \} \\
\nonumber &\leq \biggl( \sum_{i}\mathbb{E}\{|B((I_{i}^{1,n})_{\theta})|^{p}  \mid \sigma(\mathcal{T}) \}\biggr)^{1/p}\biggl( \sum_{i}\mathbb{E}\bigl\{|B^{2}(\mathcal{I}^{2,n}_{\theta-\tilde{\theta}_{n}}((I_{i}^{1,n})_{\theta})_{\tilde{\theta}_{n}-\theta}) - B^{2}(\mathcal{I}^{2,n}_{\theta-\tilde{\theta}_{n}}((I_{i}^{1,n})_{\theta})^{+})|^{q}  \mid \sigma(\mathcal{T}) \bigr\}\biggr)^{1/q} \\
&\lesssim r_{n}^{H_{1}-1/p} ((\#\mathcal{I}^{1,n})|\theta-\tilde{\theta}_{n}|^{qH_{2}})^{1/q}.
\end{align}
Plugging (\ref{ineq17}) into (\ref{eq1}) with $p=\bigl(1-\frac{H_{2}}{1-\varsigma}\bigr)^{-1}$ and $q=\bigl( \frac{H_{2}}{1-\varsigma} \bigr)^{-1}$, we obtain 
\begin{align}\label{ineq18}
\mathbb{E}\left\{  \frac{|R_{6}^{n}(\tilde{\theta_{n}})|\mathbf{1}_{A_{n}(K)}}{|D^{n}(\tilde{\theta}_{n})|} \right\} \leq \mathbb{E}\Biggl\{ \frac{r_{n}^{H_{1}+H_{2}-1+\mu + \frac{H_{2}\varsigma}{1-\varsigma} - \mu}\bigl( (\#\mathcal{I}^{1,n})|\theta-\tilde{\theta}_{n}|^{1-\varsigma} \bigr)^{\frac{H_{2}}{1-\varsigma}} }{|D^{n}(\tilde{\theta}_{n})|}\mathbf{1}_{A_{n}(K)}\Biggr\}.
\end{align}
Now we can choose sufficiently small $\mu >0$ satisfying $\frac{H_{2}\varsigma}{1-\varsigma} - \mu >0$. Applying (B3), (C3) and Jensen's inequality to (\ref{ineq18}), we complete the proof of (\ref{limit22}).

\textbf{The limit (\ref{limit23})}.
This can be proved in the same way as (\ref{limit21}). Note that in this case the term corresponding to $R_{5,2}^{n}(\tilde{\theta})$ is
\begin{align*}
\nonumber R_{7,2}^{n}(\tilde{\theta}_{n}) &:=\sum_{i,k:\overline{I_{i}^{1,n}}\leq T,\ \overline{I_{k}^{2,n}}\leq T} \langle h^{1}( (I_{i}^{1,n})_{\theta} ),  h^{2}( \mathcal{I}^{2,n}_{\theta-\tilde{\theta}_{n}}((I_{k}^{1,n})_{\theta})^{+} )  \rangle_{L^{2}([0,T+2\delta] ; \mathbb{R}^{2})} \\
\nonumber &\quad\times \langle h^{1}( (I_{k}^{1,n})_{\theta} ),  h^{2}( \mathcal{I}^{2,n}_{\theta-\tilde{\theta}_{n}}((I_{i}^{1,n})_{\theta})^{+}) \rangle_{L^{2}([0,T+2\delta] ; \mathbb{R}^{2})}.
\end{align*}
This is estimated as follows: 
\begin{align*}
\nonumber|R_{7,2}^{n}(\tilde{\theta}_{n})| &\lesssim r_{n}^{H_{1}+H_{2}} \sum_{i,k:\overline{I_{i}^{1,n}}\leq T,\ \overline{I_{k}^{2,n}}\leq T} \langle h^{1}( (I_{i}^{1,n})_{\theta} ),  h^{2}( \mathcal{I}^{2,n}_{\theta-\tilde{\theta}_{n}}((I_{k}^{1,n})_{\theta})^{+})  \rangle_{L^{2}([0,T+2\delta] ; \mathbb{R}^{2})} \\ \nonumber
&\lesssim r_{n}^{H_{1}+2H_{2}} (\#\mathcal{I}^{1,n}) \\
&= r_{n}^{2H_{1}-1+\mu} r_{n}^{2H_{2}-1+\mu} r_{n}^{2-H_{1}-2\mu}(\#\mathcal{I}^{1,n}).
\end{align*} 
This, combined with (B4), completes the proof of (\ref{limit23}). 

\textbf{The limit (\ref{limit24})}.
Finally we analyze the term $R_{8}^{n}(\tilde{\theta}_{n})$. Recall that
\begin{align*}
\nonumber |R_{8}^{n}(\tilde{\theta}_{n})| &= |\rho| c_{H_{1}}c_{H_{2}} \sum_{i:\overline{I_{i}^{n}}\leq T} \int_{ (I_{i}^{1,n})_{\theta} }du\int_{\mathcal{I}^{2,n}_{\theta-\tilde{\theta}_{n}}((I_{i}^{1,n})_{\theta})^{+}}dv\ \beta(u,v) \\&\quad \times |u-v|^{H_{1}+H_{2}-2} u^{H_{1}-H_{2}} v^{H_{2}-H_{1}}.
\end{align*}
Since $ (I_{i}^{1,n})_{\theta} \subset \mathcal{I}^{2,n}_{\theta-\tilde{\theta}_{n}}((I_{i}^{1,n})_{\theta})^{+} $ for all $i$ with $ \overline{I_{i}^{1,n}}\leq T$, it holds that 
\begin{align}\label{ineq9}
\nonumber |R_{8}^{n}(\tilde{\theta}_{n})| &\geq |\rho| c_{H_{1}}c_{H_{2}} \sum_{i:\overline{I_{i}^{1,n}}\leq T} \int_{ (I_{i}^{1,n})_{\theta} }du\int_{(I_{i}^{1,n})_{\theta}}dv\ \beta(u,v) \\&\quad \times |u-v|^{H_{1}+H_{2}-2} u^{H_{1}-H_{2}} v^{H_{2}-H_{1}}.
\end{align}
To obtain a lower bound for (\ref{ineq9}), we note the following fact: for any $\epsilon>0$, it holds that
\begin{align}\label{ineq10}
\nonumber &\quad\sum_{i:\overline{I_{i}^{1,n}}\leq T}\int_{ (I_{i}^{1,n})_{\theta} }du\int_{(I_{i}^{1,n})_{\theta}}dv\ \beta(u,v) |u-v|^{H_{1}+H_{2}-2} u^{H_{1}-H_{2}} v^{H_{2}-H_{1}} \\
\nonumber &\geq \sum_{i:\overline{I_{i}^{1,n}} \leq T, \underline{I_{i}^{1,n}} \geq \epsilon}\int_{ (I_{i}^{1,n})_{\theta} }du\int_{(I_{i}^{1,n})_{\theta}}dv\ \beta(u,v) |u-v|^{H_{1}+H_{2}-2} u^{H_{1}-H_{2}} v^{H_{2}-H_{1}}\\
&\gtrsim  \sum_{i:\overline{I_{i}^{1,n}} \leq T, \underline{I_{i}^{1,n}} \geq \epsilon}\int_{ (I_{i}^{1,n})_{\theta} }du\int_{(I_{i}^{1,n})_{\theta}}dv\ |u-v|^{H_{1}+H_{2}-2}.
\end{align}


Combining (\ref{ineq9}) nad (\ref{ineq10}), we obtain 
\begin{align}\label{ineq12} 
|R_{8}^{n}(\tilde{\theta}_{n})| &\gtrsim \sum_{i:\overline{I_{i}^{1,n}} \leq T, \underline{I_{i}^{1,n}} \geq \epsilon} | I_{i}^{1,n} |^{H_{1}+H_{2}}
\end{align}
for any $\epsilon>0$. Therefore we obtain 
\begin{align*}
\left| \frac{R_{8}^{n}(\tilde{\theta}_{n})}{D^{n}(\tilde{\theta}_{n})} \right| \gtrsim \frac{ \sum_{i:\overline{I_{i}^{1,n}} \leq T, \underline{I_{i}^{1,n}} \geq \epsilon} | I_{i}^{1,n} |^{H_{1}+H_{2}} }{ \sqrt{\sum_{i} |I_{i}^{1,n}|^{2H_{1}} } \sqrt{ \sum_{j} |I_{j}^{2,n}|^{2H_{2}} } }.
\end{align*}
We complete the proof of (\ref{limit24}) by (B2) (especially (\ref{limit7})).

As we noted above, if the true parameter value $\theta$ is positive, then $ \tilde{\theta}_{n} $ is in $ \Theta_{\geq0} $ for sufficiently large $n$ since $ |\tilde{\theta}_{n} - \theta|\leq \rho_{n} $. Therefore we can conclude (\ref{limit6}) by (\ref{limit22})-(\ref{limit24}) if $ \theta >0 $. However $\tilde{\theta}_{n}$ may be negative for any $n$ when $ \theta=0 $. In order to obtain (\ref{limit6}) when $\theta=0$, it suffices to show (\ref{limit22})-(\ref{limit24}) hold for $\theta=0$ and $ \tilde{\theta}_{n} \in [-\rho_{n},0) $. This is completely analogous to the case where $ \tilde{\theta}_{n} \in \Theta_{\geq0} $, so that we only add a few remarks and omit the proof. 

If  $\theta=0$ and $ \tilde{\theta}_{n} \in [-\rho_{n},0) $, then (\ref{calc6}) becomes
\begin{align} \label{calc7}
\nonumber R_{2}^{n}(\tilde{\theta}_{n}) &= \sum_{i,j:\overline{I_{j}^{2,n}} \leq T } B^{1}(I_{i}^{1,n}) B^{2}(I_{j}^{2,n} ) \mathbf{1}_{ (I_{i}^{1,n})_{\tilde{\theta}_{n}} \cap I_{j}^{2,n} \neq \emptyset } \\
\nonumber &= \sum_{j:\overline{I_{j}^{2,n}} \leq T } \left( B^{1}(\mathcal{I}_{\tilde{\theta}_{n}}^{1,n}(I_{j}^{2,n})_{-\tilde{\theta}_{n}}) - B^{1}(\mathcal{I}_{\tilde{\theta}_{n}}^{1,n}(I_{j}^{2,n})^{+} ) \right)B^{2}(I_{j}^{2,n}) \\
\nonumber &\quad +\sum_{j:\overline{I_{j}^{2,n}} \leq T } \left(B^{1}(\mathcal{I}_{\tilde{\theta}_{n}}^{1,n}(I_{j}^{2,n})^{+}) B^{2}(I_{j}^{2,n}) - \langle h^{1}( \mathcal{I}_{\tilde{\theta}_{n}}^{1,n}(I_{j}^{2,n})^{+} ) , h^{2}( I_{j}^{2,n} ) \rangle_{L^{2}([0,T+2\delta];\mathbb{R}^{2})} \right) \\
&\quad + \sum_{j:\overline{I_{j}^{2,n}} \leq T }  \langle h^{1}( \mathcal{I}_{\tilde{\theta}_{n}}^{1,n}(I_{j}^{2,n})^{+} ) , h^{2}( I_{j}^{2,n} ) \rangle_{L^{2}([0,T+2\delta];\mathbb{R}^{2})}.
\end{align} 
Here the symbol $  \mathcal{I}_{\tilde{\theta}_{n}}^{1,n}(I_{j}^{2,n})^{+} $ denotes $ \mathcal{I}_{\tilde{\theta}_{n}}^{1,n}(I_{j}^{2,n}) \cap [0,T+2\delta]  $ as before. 
Let $ \hat{R}_{8}^{n}(\tilde{\theta}_{n}) $ denote the last term in (\ref{calc7}). Then the inequality corresponding to (\ref{ineq12}) is 
\begin{align*}
|\hat{R}_{8}^{n}(\tilde{\theta}_{n})| &\gtrsim \sum_{j:\overline{I_{j}^{2,n}} \leq T, \underline{I_{j}^{2,n}} \geq \epsilon} | I_{i}^{2,n} |^{H_{1} + H_{2}}.
\end{align*}
In particular, we use (\ref{limit8}) of (B2) in this case. 
\end{proof}

The following corollary is immediate from Propositions \ref{lemma2} and \ref{lemma5} and Lemma \ref{lemma1}.

\begin{corollary}\label{corollary1}
\begin{asparaenum}
\item[(1)] We have 
\begin{align*}
\sup_{\tilde{\theta}\in\mathcal{G}^{n} \cap \{ \tilde{\theta} \in \Theta \mid | \tilde{\theta} - \theta | \geq \epsilon v_{n} \} }|\mathcal{U}_{n}(\tilde{\theta})| \to^{p} 0
\end{align*}
as $n\to\infty$ for any $ \epsilon >0 $.
\item[(2)] Suppose that $ \tilde{\theta}_{n} \in \mathcal{G}^{n} $ satisfies $ |\theta-\tilde{\theta}_{n}| \leq \rho_{n} $ for all $n\in\mathbb{N}$. Then there exists $ c_{\ast} >0 $ such that 
\begin{align*}
\mathbb{P} \left\{ | \mathcal{U}_{n}(\tilde{\theta}_{n}) | \geq c_{\ast} \right\} \to 1
\end{align*}
as $n\to\infty$.
\end{asparaenum}
\end{corollary}

Now we are ready to prove Theorem \ref{theorem1} under the additional assumptions $ A^{1} = A^{2} = 0 $. 

\begin{proof}[Proof of Theorem \ref{theorem1} when $A^{1}=A^{2}=0$]
Thanks to (C3), there is a sequence $ (\tilde{\theta}_{n})_{n\in\mathbb{N}} \subset \mathcal{G}^{n} $ satisfying $ |\tilde{\theta}_{n} - \theta| \leq \rho_{n} $. By the definition of $ \hat{\theta}_{n} $, we have $ |\mathcal{U}_{n}(\tilde{\theta}_{n})| \leq |\mathcal{U}_{n}(\hat{\theta}_{n})| $. On the other hand, one obtains $ |\mathcal{U}_{n}(\hat{\theta}_{n})| \leq \sup_{\tilde{\theta}\in\mathcal{G}^{n} \cap \{ \tilde{\theta}\in \Theta \mid | \tilde{\theta} - \theta | \geq \epsilon v_{n} \} }|\mathcal{U}_{n}(\tilde{\theta})| $ on the set $ \{ v_{n}^{-1}| \hat{\theta_{n}} - \theta | \geq \epsilon \} $. Therefore, we obtain
\begin{align}\label{ineq11}
\mathbb{P}\{ v_{n}^{-1}| \hat{\theta_{n}} - \theta | \geq \epsilon \} \leq \mathbb{P}\left\{ | \mathcal{U}_{n}(\tilde{\theta}_{n}) | \leq \sup_{\tilde{\theta}\in\mathcal{G}^{n} \cap \{ \tilde{\theta}\in \Theta \mid | \tilde{\theta} - \theta | \geq \epsilon v_{n} \} }|\mathcal{U}_{n}(\tilde{\theta})| \right\}.
\end{align}
The right-hand-side of (\ref{ineq11}) tends to $0$ as $n\to\infty$ by Corollary \ref{corollary1}.
\end{proof}

\subsection{Proof of Theorem \ref{theorem1}}

Let $ \mathcal{U}_{n}^{0}(\tilde{\theta}) $ denote the right-hand-side of (\ref{calc8}). Now we consider the general case: we no longer assume $A^{1}=A^{2}=0$. We reduce the general case to the previous one.  We establish a relation between $\mathcal{U}_{n}(\tilde{\theta})$ and $\mathcal{U}_{n}^{0}(\tilde{\theta})$ in Proposition \ref{corollary2} below. We begin with a lemma that relates the denominators of $\mathcal{U}_{n}(\tilde{\theta})$ and those of $\mathcal{U}_{n}^{0}(\tilde{\theta})$.

\begin{lemma}\label{lemma9}
Let $ J $ denote $ [0,T+\delta] $ or $[0,T]$. We set $a_{1}=\theta$ and $a_{2}=0$. Then we have
\begin{align} \label{limit9}
\frac{ \sum_{i:\overline{I_{i}^{n}} \in J } ( X^{l}(I_{i}^{l,n})^{2} - \sigma_{l}^{2} B^{l}((I_{i}^{l,n})_{a_{l}})^{2} ) }{ \sum_{i:\overline{I_{i}^{n}} \in J } | I_{i}^{l,n} |^{2H_{l}} } \to^{p} 0
\end{align}
as $n\to\infty$ for $l=1,2$.
\end{lemma}
\begin{proof}
By the H\"{o}lder continuity of $ B^{l} $, we have, for any $ \epsilon >0 $, 
\begin{align*}
|X^{l}(I_{i}^{l,n})^{2} - \sigma_{l}^{2} B^{l}((I_{i}^{l,n})_{a_{l}})^{2}| &= |A^{l}(I_{i}^{l,n})^{2} + 2\sigma_{l} A^{l}(I_{i}^{l,n}) B^{l}((I_{i}^{l,n})_{a_{l}}) |\\
&\lesssim | I_{i}^{l,n} |^{2H_{l}} + | I_{i}^{l,n} |^{1+H_{l}-\epsilon} \\
&\lesssim r_{n} | I_{i}^{l,n} | + r_{n}^{H_{l}-\epsilon} | I_{i}^{l,n} |.
\end{align*}
Since $ H_{l}-\epsilon = (2H_{l}-1) +1 -H_{l} -\epsilon $, we obtain $ (\ref{limit9}) $ by (B3) if we choose $\epsilon>0$ such that $ 1-H_{l}-\epsilon >0 $.
\end{proof}

We set 
\begin{align*}
\nonumber \check{R}_{2}^{n}(\tilde{\theta}) &= \mathbf{1}_{\Theta_{\geq0}}(\tilde{\theta})\sum_{i,j:\overline{I_{i}^{1,n}}\leq T} X^{1}(I_{i}^{1,n})X^{2}(I_{j}^{2,n})\mathbf{1}_{I_{i}^{1,n}\cap (I_{j}^{2,n})_{-\tilde{\theta}}\neq \emptyset} \\
&\quad + \mathbf{1}_{\Theta_{<0}}(\tilde{\theta})\sum_{i,j:\overline{I_{j}^{2,n}}\leq T} X^{1}(I_{i}^{1,n})X^{2}(I_{j}^{2,n})\mathbf{1}_{(I_{i}^{1,n})_{\tilde{\theta}}\cap (I_{j}^{2,n})\neq \emptyset}.
\end{align*}
The next lemma relates the numerators of $\mathcal{U}_{n}(\tilde{\theta})$ and those of $\mathcal{U}_{n}^{0}(\tilde{\theta})$. 

\begin{lemma}\label{lemma10}
We have 
\begin{align}\label{limit10}
 \sup_{\tilde{\theta} \in \mathcal{G}^{n} } \left| \frac{ \check{R}_{2}^{n}(\tilde{\theta}) - \sigma_{1}\sigma_{2} R_{2}^{n}(\tilde{\theta}) }{D^{n}(\tilde{\theta})} \right| \to^{p} 0
\end{align}
as $n\to\infty$.
\end{lemma}
\begin{proof}
Note that 
\begin{align*}
\nonumber X^{1}(I_{i}^{1,n}) X^{2}(I_{j}^{2,n})  &= A^{1}(I_{i}^{1,n})A^{2}(I_{j}^{2,n}) + \sigma_{2} A^{1}(I_{i}^{1,n}) B^{2}(I_{j}^{2,n}) \\ &\quad  + \sigma_{1} B^{1}((I_{i}^{1,n})_{\theta}) A^{2}(I_{j}^{2,n}) +  \sigma_{1}\sigma_{2} B^{1}((I_{i}^{1,n})_{\theta}) B^{2}(I_{j}^{2,n}).
\end{align*}
Obviously (\ref{limit10}) is equivalent to 
\begin{align}\label{limit11}
\sup_{\tilde{\theta} \in \mathcal{G}^{n} \cap \Theta_{\geq0} } \left| \frac{ \check{R}_{2}^{n}(\tilde{\theta}) - \sigma_{1}\sigma_{2}R_{2}^{n}(\tilde{\theta}) }{D^{n}(\tilde{\theta})} \right| \to^{p} 0
\end{align}
and 
\begin{align}\label{limit12}
\sup_{\tilde{\theta} \in \mathcal{G}^{n} \cap \Theta_{<0} } \left| \frac{ \check{R}_{2}^{n}(\tilde{\theta}) - \sigma_{1}\sigma_{2}R_{2}^{n}(\tilde{\theta}) }{D^{n}(\tilde{\theta})} \right| \to^{p} 0.
\end{align}
We only prove (\ref{limit11}). The proof of (\ref{limit12}) is completely analogous. Since $ \tilde{\theta} \in \Theta_{\geq0} $, we have 
\begin{align*}
\nonumber |\check{R}_{2}^{n}(\tilde{\theta}) -R_{2}^{n}(\tilde{\theta})| &\leq \sum_{i: \overline{I_{i}^{1,n}}\leq T} | A^{1}(I_{i}^{1,n}) A^{2}(\mathcal{I}_{-\tilde{\theta}}^{2,n}(I_{i}^{1,n})_{\tilde{\theta}}) | \\ \nonumber
&\quad + \sum_{i: \overline{I_{i}^{1,n}}\leq T} | \sigma_{2}A^{1}(I_{i}^{1,n}) B^{2}(\mathcal{I}_{-\tilde{\theta}}^{2,n}(I_{i}^{1,n})_{\tilde{\theta}}) | \\
 \nonumber &\quad + \sum_{j} | \sigma_{1} B^{1}( \mathcal{I}_{\theta}^{1,n,\leq T} ((I_{j}^{2,n})_{\theta-\tilde{\theta}}) ) A^{2}(I_{j}^{2,n})  | \\
\nonumber &\lesssim r_{n} + r_{n}^{H_{1}} + r_{2}^{H_{2}}\\
&= r_{n}^{H_{1}+H_{2}-1+\mu} ( r_{n}^{2-H_{1}-H_{2}-\mu} + r_{n}^{1-H_{2}-\mu} + r_{n}^{1-H_{1}-\mu} ).
\end{align*} 
We can choose $\mu>0$ such that $ 1-(H_{1}\vee H_{2}) -\mu >0 $ so that we conclude (\ref{limit11}) by (B3).
\end{proof}

Let us set $S_{n} = \bigcap_{l=1,2}\bigl\{\sum_{i:\overline{I_{i}^{l,n}}\leq T} X^{l}(I_{i}^{l,n})^{2} \neq 0\bigr\}$. Note that $\mathbb{P}\{S_{n}\}\to 1$ as $n\to\infty$ by Lemmas \ref{lemma1} and \ref{lemma9}. The next proposition shows that we can reduce calculation of $ \mathcal{U}_{n}(\tilde{\theta}) $ to that of $ \mathcal{U}_{n}^{0}(\tilde{\theta}) $. 

\begin{proposition}\label{corollary2}
It holds that
\begin{align}\label{calc9}
\mathcal{U}_{n}(\tilde{\theta}) = \left(\bar{o}_{p}(1)(\tilde{\theta}) + \frac{1+o_{p}(1)}{1+o_{p}(1)} \mathcal{U}_{n}^{0}(\tilde{\theta})\right)\mathbf{1}_{S_{n}}.
\end{align} 
Here $ \bar{o}_{p}(1)(\tilde{\theta}) $ denotes a quantity such that
\begin{align*}
\sup_{\tilde{\theta}\in\mathcal{G}^{n}} \bar{o}_{p}(1)(\tilde{\theta}) \to^{p} 0.
\end{align*}
In particular, 
\begin{enumerate}
\item[(1)] we have 
\begin{align*}
\sup_{\tilde{\theta}\in\mathcal{G}^{n}\cap \{ \tilde{\theta}\in \Theta \mid |\tilde{\theta}-\theta|\geq\epsilon v_{n} \}} | \mathcal{U}_{n}(\tilde{\theta}) | \to^{p} 0
\end{align*}
as $n\to\infty$ for any $ \epsilon >0 $, and
\item[(2)] if $ \tilde{\theta}_{n} \in \mathcal{G}^{n} $ satisfies $ |\tilde{\theta}_{n} -\theta| \leq v_{n}^{1+\alpha} $, then there exists $c_{\ast} >0$ such that
\begin{align*}
\mathbb{P} \left\{ |\mathcal{U}_{n}(\tilde{\theta}_{n})| \geq c_{\ast} \right\} \to 1
\end{align*}
as $n\to\infty$.
\end{enumerate}
\end{proposition}
\begin{proof}
We have (\ref{calc9}) by Lemmas \ref{lemma1}, \ref{lemma9} and \ref{lemma10}.
\end{proof}

Now we are ready to prove Theorem \ref{theorem1}.

\begin{proof}[Proof of Theorem \ref{theorem1}]
Thanks to Proposition \ref{corollary2}, (\ref{limit13}) follows by the same line of argument as in the proof of Theorem $\ref{theorem1}$ when $ A^{1}=A^{2}=0 $.
\end{proof}

\section{Examples} \label{section7}

In this section, we give some examples that satisfy the assumptions of Theorem \ref{theorem1}.

\subsection{Processes} \label{section3}

\subsubsection{The solution of a stochastic differential equation driven by fBM}

Let $(\Omega,\mathcal{F},\mathbb{G}=(\mathcal{G}_{t})_{t\in[0,T+2\delta]},\mathbb{P})$ be a stochastic basis which supports fBMs $B^{1}$ and $B^{2}$ satisfying (\ref{calc4}). Here we assume that the $\sigma$-field $ \mathcal{F} $ is complete and the filtration $\mathbb{G}$ satisfies the usual conditions.  Let us consider a stochastic differential equation 
\begin{align}\label{equation1}
X_{t}^{l} = X_{0}^{l} + \int_{0}^{t} b(s,X_{s}^{l})\ ds + \sigma^{l} B_{t}^{l}
\end{align}
for $ t\in[0,T+2\delta] $ and $l=1,2$. We assume that $X_{0}$ is a real-valued random variable, $\sigma^{l} \in \mathbb{R}\setminus\{0\}$ and the function $b$ satisfies the following conditions (D1) and (D2).

\begin{assumption} \label{ass4}
The function $ b\colon [0,T+2\delta]\times \mathbb{R} \to \mathbb{R} $ satisfies that
\begin{enumerate}
\item[(D1)] for each positive integer $N$ there exists a positive constant $ L_{N}>0 $ such that
\begin{align*}
|b(t,x) - b(t,y)| \leq L_{N}|x-y|
\end{align*}
for all $x\in[-N.N]$, $y\in[-N,N]$ and $t\in[0,T+2\delta]$, and
\item[(D2)] there exists a positive constant $L_{0}>0$ such that
\begin{align*}
|b(t,x)| \leq L_{0} (1+|x|)
\end{align*}
for all $ x\in\mathbb{R} $ and $t\in[0,T+2\delta]$.
\end{enumerate}
\end{assumption}

Then there exists a unique solution $X^{l}$ to the equation (\ref{equation1}) which has $1/2$-H\"{o}lder continuous sample paths almost surely. This is a special case of Theorem 2.1 of \cite{nualart2002differential}.

Let us check the processes $\tau_{-\theta}(X^{1})$ and $X^{2}$ satisfy Assumption \ref{ass1}. Since $X^{l}$ satisfies the equation (\ref{equation1}), it suffices to verify (A3). We have
\begin{align*}
| b(s,{X}_{s}^{l}) | \leq L_{0} (| {X}_{s}^{l} | + 1) \leq L_{0}( |X_{s}^{l}-X_{0}^{l}| +| {X}_{0}^{l} | + 1 )  \leq L_{0}( C(T+2\delta)^{1/2} +| {X}_{0}^{l} | + 1 )
\end{align*}
for some H\"{o}lder constant $C>0$. Hence it holds that
\begin{align*}
|A_{t}^{l} - A_{s}^{l}| \leq  L_{0}( C(T+2\delta)^{1/2} +| {X}_{0}^{l} | + 1 )|t-s|
\end{align*}
for $0\leq s\leq t \leq T+2\delta$.

\subsection{Sampling scheme} \label{section4}

\subsubsection{Synchronous and evenly spaced observations}

Let us check that the synchronous and evenly spaced observations satisfy Assumption \ref{ass2}. Let $i^{l,n}_{\ast}(a)$ denote
\begin{align*}
i^{l,n}_{\ast}(a) = \inf\{ i \mid \overline{I_{i}^{l,n}} \geq a \} 
\end{align*}
for $a>0$. Note that, using $ i^{l,n}_{\ast}(a) $, we obtain the following relation:
\begin{align*}
\frac{ \sum_{i:\overline{I_{i}^{l,n}} \leq T, \underline{I_{i}^{l,n}} \geq \epsilon} | I_{i}^{l,n} |^{H_{1}+H_{2}} }{ \sqrt{\sum_{i} |I_{i}^{1,n}|^{2H_{1}} } \sqrt{ \sum_{j} |I_{j}^{2,n}|^{2H_{2}} } } = \frac{ \sum_{i=i^{l,n}_{\ast}(\epsilon)+1}^{i^{l,n}_{\ast}(T)-1} | I_{i}^{l,n} |^{H_{1}+H_{2}} }{ \sqrt{\sum_{i} |I_{i}^{1,n}|^{2H_{1}} } \sqrt{ \sum_{j} |I_{j}^{2,n}|^{2H_{2}} } }.
\end{align*}

\begin{proposition}\label{proposition1}
Let us set 
\begin{align*}
\mathcal{T}^{1,n} = \mathcal{T}^{2,n} = \left\{ \frac{i}{n}(T+\delta) \mid i=0,\ldots,n \right\}.
\end{align*}
Then $\mathcal{T}^{1,n}$ and $\mathcal{T}^{2,n}$ satisfy Assumption \ref{ass2} with $ v_{n} = (1/n)^{1-\xi} $ for any $ \xi \in  \left( 0, \frac{1-(H_{1}\vee H_{2})}{2-(H_{1} \vee H_{2})} \right) $.
\end{proposition}
\begin{proof}
First (B1) is clear. Note that $ i_{\ast}^{l,n}(\epsilon) = \left\lceil \frac{\epsilon n}{T+\delta} \right\rceil $ in this case. Let us check (B2):
\begin{align*}
\nonumber \frac{ \sum_{i:\overline{I_{i}^{l,n}} \leq T, \underline{I_{i}^{l,n}} \geq \epsilon} | I_{i}^{l,n} |^{H_{1}+H_{2}} }{ \sqrt{\sum_{i} |I_{i}^{1,n}|^{2H_{1}} } \sqrt{ \sum_{j} |I_{j}^{2,n}|^{2H_{2}} } } &= \frac{ \sum_{i=i^{l,n}_{\ast}(\epsilon)+1}^{i^{l,n}_{\ast}(T)-1} | I_{i}^{l,n} |^{H_{1}+H_{2}} }{ \sqrt{\sum_{i} |I_{i}^{1,n}|^{2H_{1}} } \sqrt{ \sum_{j} |I_{j}^{2,n}|^{2H_{2}} } }
\\ \nonumber &= \frac{1}{n}\left( \left\lceil \frac{Tn}{T+\delta} \right\rceil - 1 - \left\lceil \frac{\epsilon n}{T+\delta} \right\rceil \right) \\
&\to \frac{T-\epsilon}{ T+\delta }
\end{align*}
as $n\to\infty$. Next we verify (B3):
\begin{align*}
\frac{ r_{n}^{2H_{l}-1+\mu} }{ \sum_{i:\overline{I_{i}^{l,n}} \leq T } | I_{i}^{l,n} |^{2H_{l}} } = \left( \frac{1}{n} \right)^{\mu} \left( \frac{1}{n}\left\lceil \frac{Tn}{T+\delta} \right\rceil \right)^{-1} = v_{n}^{\frac{\mu}{1-\xi}} \left( \frac{1}{n}\left\lceil \frac{Tn}{T+\delta} \right\rceil \right)^{-1} .
\end{align*}
We can take $ \gamma(\mu) = \mu/(1-\xi)$. 
Finally we verify (B4):
\begin{align*}
v_{n}^{2-(H_{1}\vee H_{2})} \mathbb{E} \{ \#\mathcal{I}^{1,n} \} =  v_{n}^{2-(H_{1}\vee H_{2})} \mathbb{E} \{ \# \mathcal{I}^{2,n} \} = \left(\frac{1}{n}\right)^{(1-\xi)(2-(H_{1}\vee H_{2})) -1}.
\end{align*}
If $ \xi \in \left(0,\frac{1-(H_{1}\vee H_{2})}{2-(H_{1}\vee H_{2})} \right) $, then $(1-\xi)(2-(H_{1}\vee H_{2})) -1 >0$ so that (B4) holds.
\end{proof}

\subsubsection{Poisson Sampling}

We check that the Poisson sampling scheme satisfies Assumption \ref{ass2}. Let $ p_{1} >0 $ and  $ p_{2} >0 $ be positive numbers, and $N^{1,n} $ and $ N^{2,n}$ be mutually independent Poisson processes with intensity $ p_{1}n $ and $ p_{2}n $, respectively. We also assume $N^{1,n}$ and $N^{2,n}$ are independent of the underlying fractional Brownian motions $B^{1}$ and $B^{2}$. We set $ T_{i}^{l,n} = \inf\{ t\geq 0\colon N_{t}^{l,n} = i \} \wedge (T+\delta) $, $  \tilde{T}_{i}^{l,n} = \inf\{ t\geq 0\colon N_{t}^{l,n} = i \}  $, $I_{i}^{l,n} = ( T_{i-1}^{l,n} , T_{i}^{l,n} ]$, $ \tilde{I}_{i}^{l,n} = ( \tilde{T}_{i-1}^{l,n}, \tilde{T}_{i}^{l,n} ] $ and $ \lambda^{l,n}(t) = \lceil p_{l}nt \rceil  $. Note that $i_{\ast}^{l,n}(t) = N_{t-}^{l,n}+1$ in this case.

\begin{proposition}\label{proposition2}
We set
\[
\mathcal{T}^{1,n} = \left\{ T_{i}^{1,n} \mid i=0,\ldots,i_{\ast}^{1,n}(T+\delta) \right\}
\]
and
\[
\mathcal{T}^{2,n} = \left\{ T_{j}^{2,n} \mid j=0,\ldots,i_{\ast}^{2,n}(T+\delta) \right\}.
\]
Then $\mathcal{T}^{1,n}$ and $\mathcal{T}^{2,n}$ satisfy Assumption \ref{ass2} with $v_{n} = (1/n)^{1-\eta}$ for any $\eta \in \left( 0, \frac{1-(H_{1}\vee H_{2})}{2-(H_{1} \vee H_{2})} \right)$.
\end{proposition}

In order to prove Proposition \ref{proposition2}, we use the following lemmas.

\begin{lemma}\label{hy8}
Let 
\begin{align*} \tilde{r}_{n} = \left(\max_{ 1\leq i \leq \lambda^{1,n}(T+\delta)\vee i_{\ast}^{1,n}(T+\delta) } | \tilde{I}_{i}^{1,n} | \right) \vee \left(\max_{ 1\leq j \leq \lambda^{2,n}(T+\delta)\vee i_{\ast}^{2,n}(T+\delta) } | \tilde{I}_{j}^{2,n} | \right). 
\end{align*}
Then for any $ q\geq 1 $, we have
\begin{align*}
E\{ \tilde{r}_{n}^{q} \}  = o(n^{\alpha}),
\end{align*}
for any $ 0 < \alpha < q $.
\end{lemma}
\begin{proof}
See Lemma 8 of \cite{hayashi2008asymptotic}.
\end{proof}

\begin{lemma}\label{hy9}
Let $\epsilon \in (1/2,1) $. Then, for each $T>0$, it holds that
\begin{align*}
\mathrm{ P }\{ | \lambda^{l,n}(T) - i_{\ast}^{l,n}(T) | \leq n^{\epsilon} \} \to 1,
\end{align*}
as $n\to\infty$ for $l=1,2$.
\end{lemma}
\begin{proof}
See Lemma 9 of \cite{hayashi2008asymptotic}.
\end{proof}

Before starting the proof of Proposition \ref{proposition2}, we introduce some notation: for $ a\in(0,T+\delta] $ and $ H\in(1/2,1) $, we set 
\begin{align*}
R_{9}^{l,n}(a)= \frac{ \sum_{i=1}^{i_{\ast}^{l,n}(a)-1} |I_{i}^{l,n}|^{H_{1}+H_{2}} }{ \sqrt{ \sum_{i} | I_{i}^{1,n} |^{2H_{1}} }\sqrt{\sum_{j} | I_{j}^{2,n} |^{2H_{2}}}  },
\end{align*}
\begin{align*}
R_{10}^{l,n,\pm}(a) = \frac{ \sum_{i=1}^{\lambda^{l,n}(a) \pm \lceil 2n^{2/3} \rceil} |\tilde{I}_{i}^{l,n}|^{H_{1}+H_{2}} }{ \sqrt{ \sum_{i=1}^{\lambda^{l,n}(T+\delta) \mp \lceil 2n^{2/3} \rceil} | \tilde{I}_{i}^{1,n} |^{2H_{1}} }\sqrt{\sum_{j=1}^{\lambda^{l,n}(T+\delta) \mp \lceil 2n^{2/3} \rceil} | \tilde{I}_{j}^{2,n} |^{2H_{2}}}  }, 
\end{align*}
\begin{align*}
R_{11}^{l,n,\pm}(H,a) = \sum_{i=1}^{\lambda^{l,n}(a) \pm \lceil 2n^{2/3}\rceil }  \left( |\tilde{I}_{i}^{l,n}|^{2H} - \frac{\Gamma(2H+1)}{(np_{l})^{2H}} \right),
\end{align*}
and
\begin{align*}
R_{12}^{l,n,\pm}(H,a) = \frac{\Gamma(2H+1)(\lambda^{l,n}(a)\pm\lceil2n^{2/3}\rceil)}{(np_{l})^{2H} }.
\end{align*}
Note that we have the following relation:
\begin{align}\label{calc10}
\nonumber R_{10}^{l,n,\pm}(a) &= \left( R_{11}^{l,n,\pm} \left(\frac{H_{1}+H_{2}}{2} , a\right) + R_{12}^{l,n,\pm} \left(\frac{H_{1}+H_{2}}{2} , a\right) \right) \\ 
\nonumber &\quad \times \left( R_{11}^{1,n,\mp}(H_{1} , T+\delta) + R_{12}^{1,n,\mp}(H_{1} , T+\delta) \right)^{-1/2} \\
&\quad \times \left( R_{11}^{2,n,\mp}(H_{2} , T+\delta) + R_{12}^{2,n,\mp}(H_{2} , T+\delta) \right)^{-1/2}  .
\end{align}

\begin{proof}[Proof of Proposition \ref{proposition2}]

By assumption, (B1) holds. Let us show that Poisson sampling scheme satisfies condition (B2).

\textbf{Claim.} Let $ a\in(0,T+\delta] $. Then it holds that
\begin{align} \label{limit14}
R_{9}^{n}(a) \to^{p} \frac{a}{T+\delta} c_{l}
\end{align}
as $n\to\infty$. Here $ c_{l} $ is a constant depending on $l$, $H_{1}$, $H_{2}$, $p_{1}$ and $p_{2}$:
\begin{align*}
c_{l} = \frac{p_{1}^{H_{1}-1/2} p_{2}^{H_{2}-1/2} \Gamma(H_{1}+H_{2}+1) }{p_{l}^{H_{1}+H_{2}-1 } \Gamma(2H_{1}+1) \Gamma(2H_{2}+1)}.
\end{align*}
In particular, (B2) holds. 

Let us show (\ref{limit14}). First note that a straightforward calculation yields $ \mathrm{E}\{ | \tilde{I}_{i}^{l,n} |^{2H} \} = \Gamma(2H+1)/(np_{l})^{2H} $ and $ \mathrm{Var}\{  |\tilde{I}_{i}^{l,n}|^{2H} \} = ( \Gamma(4H+1)-\Gamma(2H+1)^{2} )/(np_{l})^{4H} $ for $ H\in(1/2,1) $. 
By Lemma \ref{hy9}, we can assume $ | \lambda^{l,n}(a) - i_{\ast}^{l,n}(a) | \leq n^{\epsilon} $ and $ | \lambda^{l,n}(T+\delta) - i_{\ast}^{l,n}(T+\delta) | \leq n^{\epsilon} $ for $\epsilon\in(1/2,1)$. Since $ | \lambda^{l,n}(a) - i_{\ast}^{l,n}(a) | \leq n^{\epsilon} $, we have $ \lambda^{l,n}(a) - \lceil 2n^{\epsilon} \rceil < i^{l,n}_{\ast}(a) $ and hence $ I_{i}^{l,n} = \tilde{I}_{i}^{l,n} $ if $ i \leq \lambda^{l,n}(a) - \lceil 2n^{\epsilon} \rceil $. Note that clearly $ | I_{i}^{l,n} | \leq | \tilde{I}_{i}^{l,n} | $ holds for all $i$. Hence we have
\begin{align*}
R_{10}^{l,n,-}(a) \leq R_{9}^{l,n}(a) \leq R_{10}^{l,n,+}(a).
\end{align*}
In order to prove (\ref{limit14}), it suffices to show 
\begin{align}\label{limit15}
R_{10}^{l,n,\pm}(a) \to^{p} \frac{a}{T+\delta} c_{l}
\end{align}
as $n\to\infty$. Thanks to (\ref{calc10}), we can further reduce $ (\ref{limit15}) $ to the following limits: 
\begin{align}\label{limit16}
n^{2H-1} R_{12}^{l,n,\pm}(H,a) \to \frac{ a \Gamma(2H+1) }{(T+\delta)p_{l}^{2H-1}}
\end{align}
and 
\begin{align}\label{limit17}
R_{11}^{l,n,\pm}(H,a) = o_{p}(n^{1-2H})
\end{align}
as $n\to\infty$ for each $ H\in(1/2,1) $. First (\ref{limit16}) is straightforward. Next, Kolmogorov's inequality shows that
\begin{align}\label{ineq13}
\mathbb{P}\left\{ n^{2H-1}\left| R_{11}^{l,n,\pm} \right| \geq n^{-\xi} \right\} \leq \frac{\lambda^{l,n}(a) \pm \lceil 2n^{2/3}\rceil }{n^{2(1-\xi)}p_{l}^{4H}} (\Gamma(4H+1)-\Gamma(2H+1)^{2}) 
\end{align}
holds for $H\in(1/2,1)$, and the right-hand-side of (\ref{ineq13}) goes to $0$ as $n\to \infty$ if $ \xi<1/2 $. Hence (\ref{limit17}) holds.

As a consequence of (\ref{limit14}), we can conclude that (B2) holds: we have
\begin{align*}
\frac{ \sum_{i:\overline{I_{i}^{l,n}} \leq T, \underline{I_{i}^{l,n}} \geq \epsilon} | I_{i}^{l,n} |^{H_{1}+H_{2}} }{ \sqrt{\sum_{i} |I_{i}^{1,n}|^{2H_{1}} } \sqrt{ \sum_{j} |I_{j}^{2,n}|^{2H_{2}} } } \to^{p} \frac{T-\epsilon}{T+\delta}c_{l}
\end{align*}
as $n\to\infty$.

Next we show that (B3) holds. Note that for each $ \eta >0 $, we can choose $ \zeta>0 $ satisfying $ \eta > \zeta/(1+\zeta) $.
\textbf{Claim.} If $ \zeta >0 $ satisfies $ \eta > \zeta/(1+\zeta) $, then, for each $\mu>0$ and $\gamma(\mu):=(1+\zeta)\mu >\mu$, it holds that 
\[
\frac{ r_{n}^{2H_{l}-1+\mu} }{ \sum_{i:\overline{I_{i}^{l,n}} \leq T } | I_{i}^{l,n} |^{2H_{l}} } = o_{p}(v_{n}^{\gamma(\mu)}).
\]
In particular, (B3) holds.

Indeed, let $ \epsilon >0 $ be any positive number. Then we have
\begin{align*}
\mathbb{P}\left\{ \frac{ r_{n}^{2H_{l}-1+\mu} v_{n}^{-\gamma(\mu)} }{ \sum_{i:\overline{I_{i}^{l,n}} \leq T } | I_{i}^{l,n} |^{2H_{l}} }  \geq \epsilon \right\} &\leq \mathbb{P}\left\{ \frac{ \tilde{r}_{n}^{2H_{l}-1+\mu} n^{(1-\eta)(1+\zeta)\mu} }{ \sum_{i=1}^{\lambda^{l,n}(T)-\lceil 2n^{2/3}\rceil} | \tilde{I}_{i}^{l,n} |^{2H_{l}} } \geq \epsilon,\  | \lambda^{l,n}(T) - i_{\ast}^{l,n}(T) | \leq n^{2/3}  \right\} \\
&\quad + \mathbb{P}\left\{  | \lambda^{l,n}(T) - i_{\ast}^{l,n}(T) | > n^{2/3}  \right\} \\
&\leq \mathbb{P} \left\{ \frac{\tilde{r}_{n}^{2H_{l}-1+\mu} n^{2H_{l}-1+(1-\eta)(1+\zeta)\mu } }{ n^{2H_{l}-1}  R_{11}^{l,n,\pm}(H_{l},T) + n^{2H_{l}-1}  R_{12}^{l,n,\pm}(H_{l},T) } \geq \epsilon \right\} + o(1).
\end{align*}
Hence it holds that 
\begin{align} \label{ineq14}
\nonumber \mathbb{P}\left\{ \frac{ r_{n}^{2H_{l}-1+\mu} v_{n}^{-\gamma(\mu)} }{ \sum_{i:\overline{I_{i}^{l,n}} \leq T } | I_{i}^{l,n} |^{2H_{l}} }  \geq \epsilon \right\} &\leq \mathbb{P} \left\{ \tilde{r}_{n}^{2H_{l}-1+\mu} n^{2H_{l}-1+(1-\eta)(1+\zeta)\mu } \geq \frac{\epsilon T \Gamma(2H_{l}+1)}{2(T+\delta)p_{l}^{2H_{l}-1}} \right\} + o(1) \\
&\leq \left( \frac{\epsilon T \Gamma(2H_{l}+1)}{2(T+\delta)p_{l}^{2H_{l}-1}} \right)^{\frac{-1}{2H_{l}-1+\mu}} \mathbb{E}\{ \tilde{r}_{n} \} n^{\frac{2H_{l}-1 + (1+\zeta)(1-\eta)\mu}{2H_{l}-1+\mu} } + o(1).
\end{align}
Since $ (1+\zeta)(1-\eta)\mu < \mu $, Lemma \ref{hy8} implies that the right-hand-side of (\ref{ineq14}) goes to zero as $n\to\infty$.


Note that $\mathbb{E} \{ \#\mathcal{I}^{l,n} \} = p_{l}n(T+\delta)$. We can verify (B4) easily: since $ (1-\eta)(2-(H_{1}\vee H_{2})) >1 $ when $\eta \in \left( 0, \frac{1-(H_{1}\vee H_{2})}{2-(H_{1} \vee H_{2})} \right)$, the condition (B4) holds.
\end{proof}

\subsection{Simulation study}

In this subsection, we present a numerical simulation. We use the R package YUIMA to generate the data, plot the paths, and calculate the estimated values.

%


We consider the following setting. We set $dX^{1}_{t}=dB^{1}_{t}$ and $dX^{2}_{t}=dB^{2}_{t}$ where $B^{1}$ and $B^{2}$ satisfy (\ref{calc4}) with $H_{1}=0.6$, $H_{2}=0.7$. We vary the correlation parameter $\rho$ as $\rho=0.25, 0.50$ and $0.75$. The time horizon $T$ equals $1$. We assume that the processes $X^{1}$ and $X^{2}$ are observed by independent Poisson sampling schemes whose intensities are both $n$. We also vary the intensity $n$ as $n=100, 300$ and $500$. We set the true lead-lag parameter $\theta=0.02$, so that the process $X^{1}$ has a lead of $\theta=0.02$ over the process $X^{2}$. To calculate the estimator, we use the grids $\mathcal{G}_{1}=\{ -1 + 10^{-3}k\mid k\in\mathbb{Z}_{\geq0} \}\cap[-1,1]$ and $\mathcal{G}_{2}=\{ -1 + (3\cdot10^{-3})k\mid k\in\mathbb{Z}_{\geq0} \}\cap[-1,1]$. Hence the grid $\mathcal{G}_{1}$ contains the true parameter $\theta=0.02$, and the grid $\mathcal{G}_{2}$ does not. 

First we fix the grid. Then we repeat 300 simulations and compute the value of the estimator $\hat{\theta}_{n}$ each time, letting $\rho$ and $n$ vary as we explained above. That is, we consider the following nine cases after we fix the grid: $(\rho, n) = (0.25, 100)$, $(0.25,300)$, $(0.25, 500)$, $(0.50, 100)$, $(0.50, 300)$, $(0.50, 500)$, $(0.75,100)$, $(0.75, 300)$ and $(0.75,500)$. 

Figure \ref{figure1} (resp. 2) shows a plot of estimated values $\hat{\theta}_{n}$ with the grid $\mathcal{G}_{1}$ (resp. $\mathcal{G}_{2}$). The red dashed lines show the true parameter $\theta=0.02$. 

Figure \ref{figure3} shows a simulated observation data. Figure \ref{figure4} shows a plot of the contrast function $\mathcal{U}_{n}(\tilde{\theta})$ using the simulated data shown in Figure \ref{figure3}. We use the grid $\mathcal{G}=\mathcal{G}_{1}$. Again the red dashed line shows the true parameter $\theta=0.02$. In this simulation, the estimated parameter value is $\hat{\theta}_{n}=0.02$.

\section*{Acknowledgement}
The author would like to express deepest gratitude to Professor Nakahiro Yoshida for introducing him to this problem and for many valuable suggestions. He is also very grateful to two anonymous referees for their really helpful comments and suggestions. He could never have improved the results and the presentation of this paper without their helpful comments. 

This work was supported by JST CREST and the Program for Leading Graduate Schools, MEXT, Japan.

\section*{Conflict of interest}
The author has no conflicts of interest to declare.

\bibliographystyle{abbrvnat}
\bibliography{fllagbib} 

%

\begin{figure}
\begin{center} 
\includegraphics[clip]{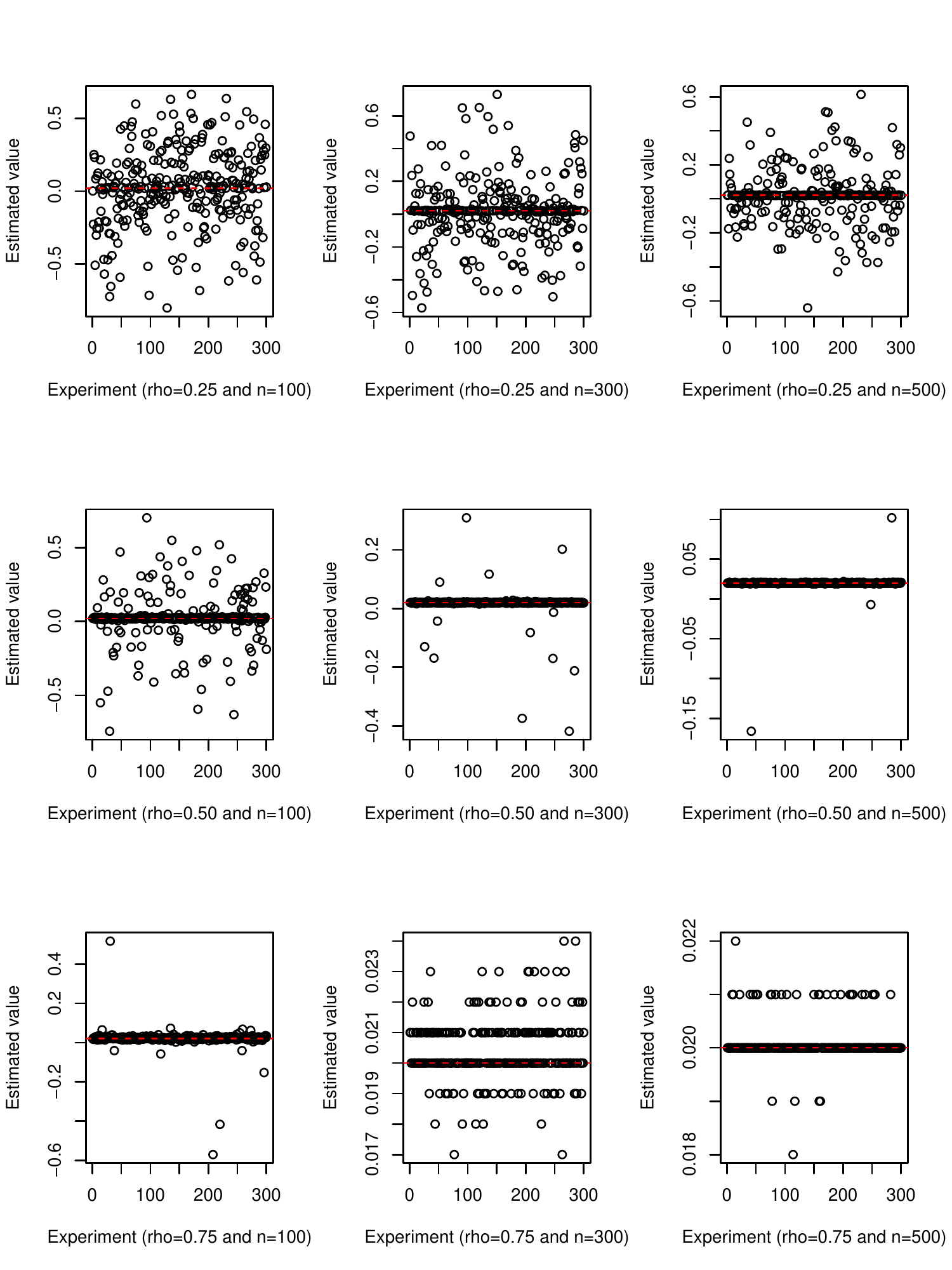}
\caption{A plot of estimated values $\hat{\theta}_{n}$ (colored online). We fix a grid $\mathcal{G}_{1}=\{ -1 + 10^{-3}k\mid k\in\mathbb{Z}_{\geq0} \}\cap[-1,1]$, repeat 300 simulations and plot the computed values of the estimator $\hat{\theta}_{n}$. The red dashed lines show the true parameter $\theta=0.02$.}
\label{figure1}
\end{center}
\end{figure}

\begin{figure}
\begin{center} 
\includegraphics[clip]{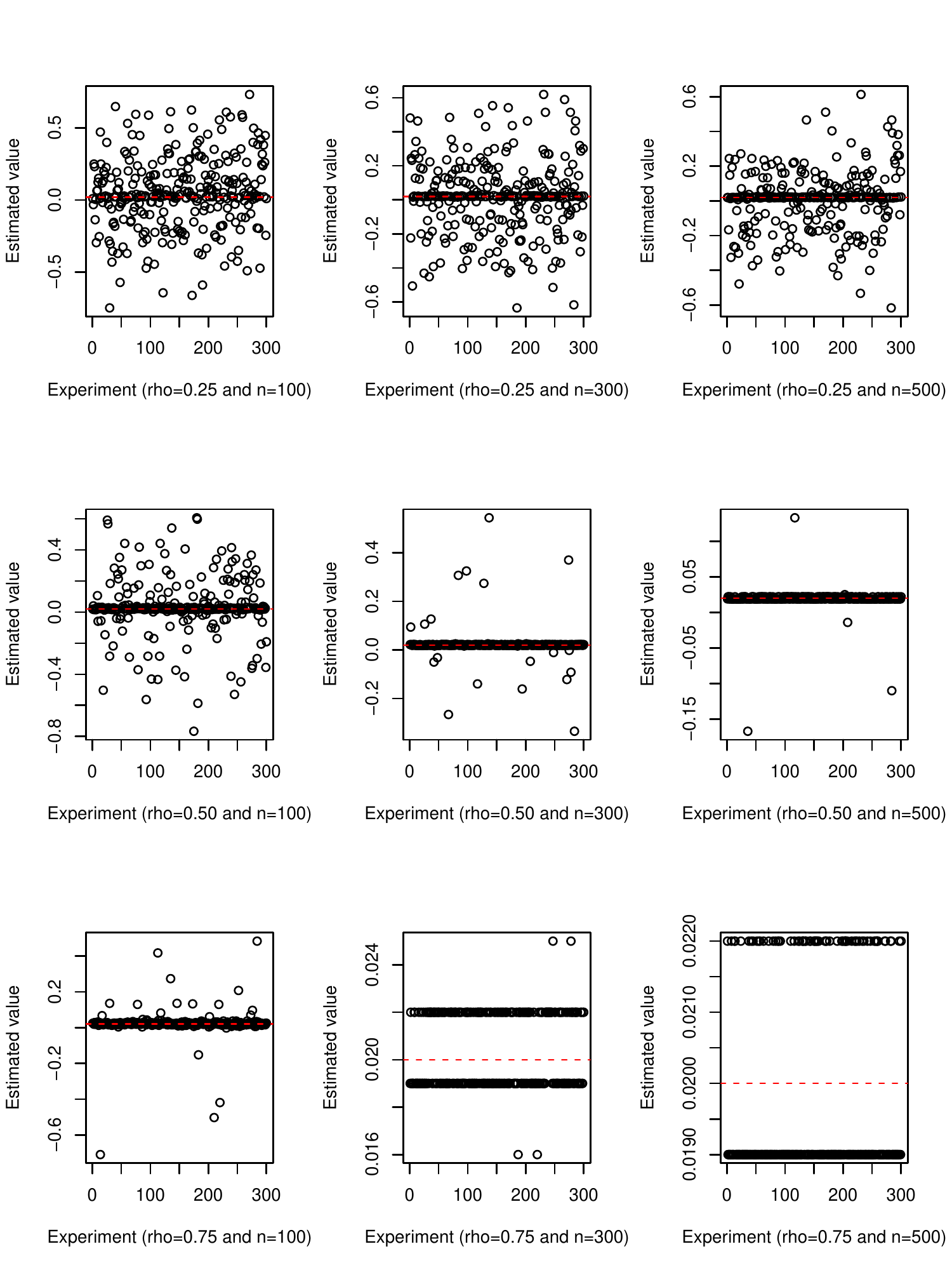}
\caption{A plot of estimated values $\hat{\theta}_{n}$ (colored online). We fix a grid $\mathcal{G}_{2}=\{ -1 + (3\cdot10^{-3})k\mid k\in\mathbb{Z}_{\geq0} \}\cap[-1,1]$, repeat 300 simulations and plot the computed values of the estimator $\hat{\theta}_{n}$. The red dashed lines show the true parameter $\theta=0.02$. Note that the grid $\mathcal{G}_{2}$ does not contain the true parameter $\theta=0.02$.}
\label{figure2}
\end{center}
\end{figure}

\begin{figure}
\begin{center} 
\includegraphics[clip]{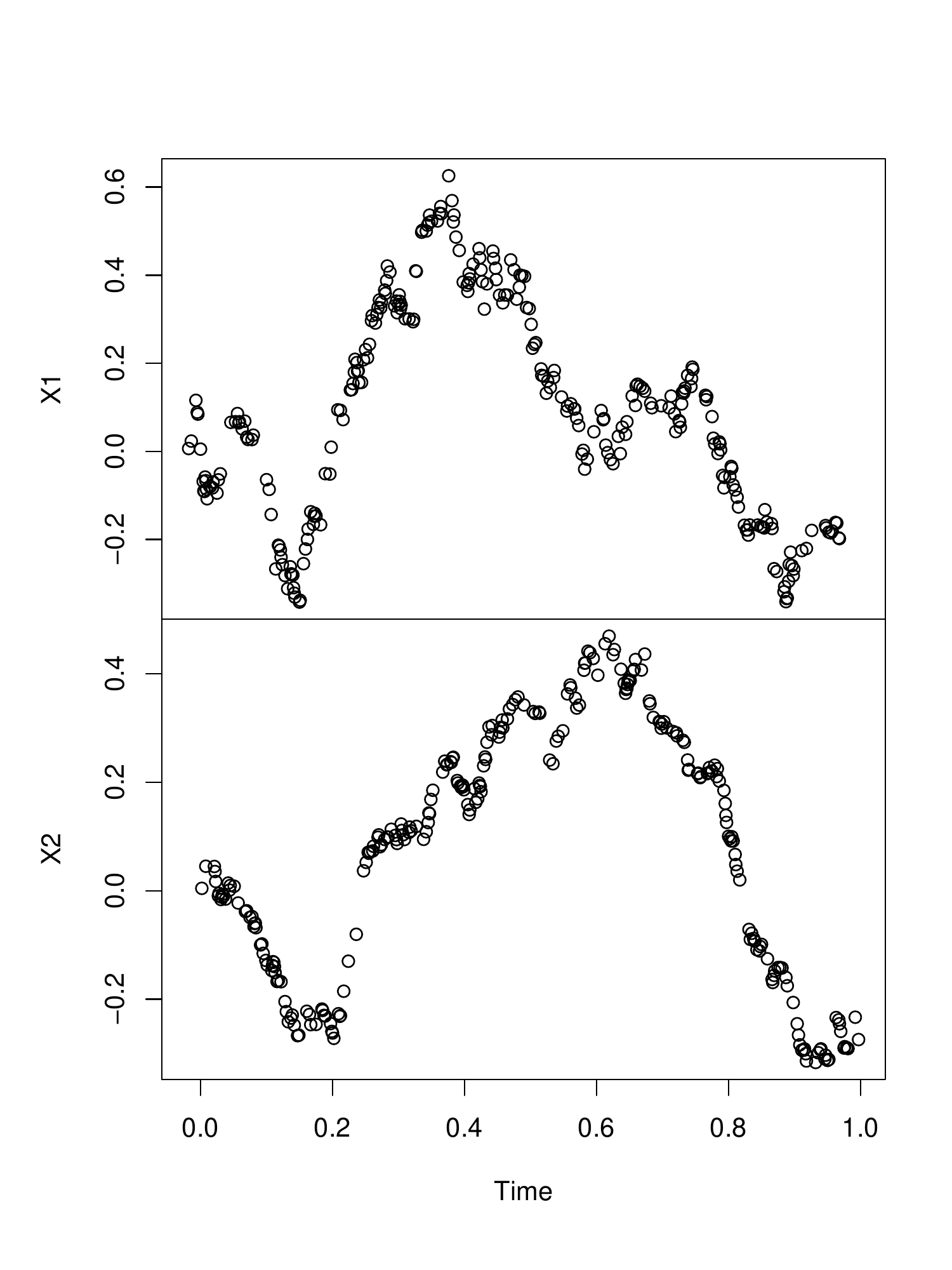}
\caption{A simulated observation data. The underlying processes are $dX^{1}_{t}=dB^{1}_{t}$ and $dX^{2}_{t}=dB^{2}_{t}$ where $B^{1}$ and $B^{2}$ satisfy (\ref{calc4}) with $H_{1}=0.6$, $H_{2}=0.7$ and $\rho=0.50$. The time horizon is $T=1$, and we use independent Poisson sampling schemes both with intensity $n=300$. The lead-lag parameter is $\theta = 0.02$, so that the process $X^{1}$ has a lead of $\theta=0.02$ over the process $X^{2}$. } 
\label{figure3}
\end{center}
\end{figure}

\begin{figure}
\begin{center} 
\includegraphics[clip]{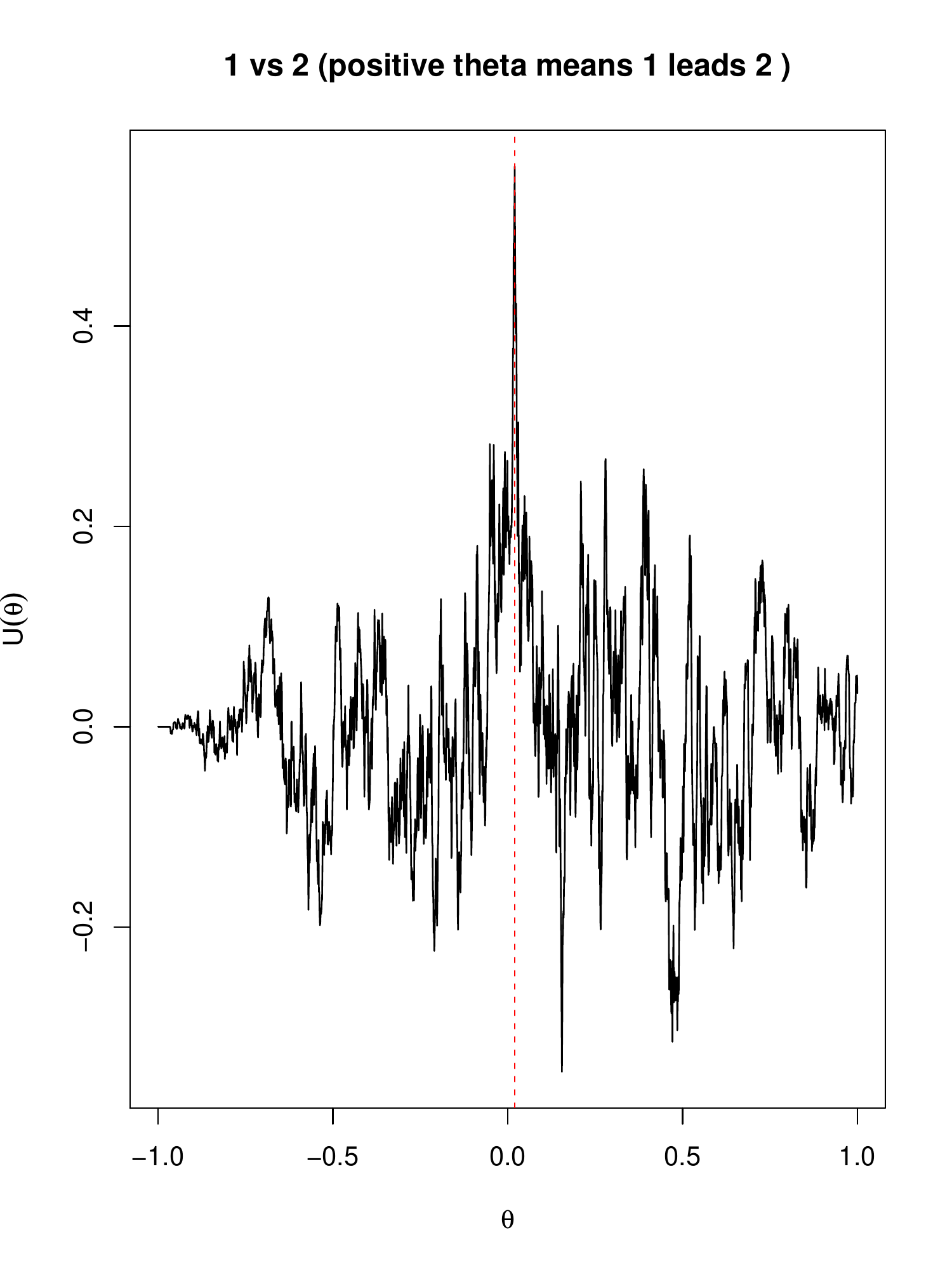}
\caption{A plot of the contrast function $\mathcal{U}_{n}(\tilde{\theta})$ (colored online). We use the simulated data shown in Figure \ref{figure3}. In order to maximize the contrast, we use the grid $\mathcal{G}_{1}=\{ -1 + 10^{-3}k\mid k\in\mathbb{Z}_{\geq0} \}\cap[-1,1]$. The red dashed vertical line shows $\tilde{\theta}=0.02$. In this simulation, the estimated value is $\hat{\theta}_{n}=0.02$.  } 
\label{figure4}
\end{center}
\end{figure}

\end{document}